\documentclass[11pt]{amsart}
\pdfoutput=1

\usepackage{amssymb,latexsym}
\usepackage{amsmath,amsfonts,mathtools,amsthm}
\usepackage{enumerate}
\usepackage{enumitem}
\usepackage[T1]{fontenc}
\usepackage[latin1]{inputenc}
\usepackage[english]{babel}

\usepackage{mathrsfs}

\usepackage{inputenc}

\usepackage{graphicx}
\usepackage{texdraw}

\usepackage[bookmarks=true,%
    colorlinks=true,%compo
    linkcolor=blue,%
    citecolor=blue,%
    filecolor=blue,%
    menucolor=blue,%
    urlcolor=blue,%
    breaklinks=true]{hyperref}

\usepackage{caption}

\usepackage[all]{xy}
\CompileMatrices
\def\cxymatrix#1{\xy*[c]\xybox{\xymatrix#1}\endxy}

\theoremstyle{plain}
\newtheorem{theorem}{Theorem}[section]
\theoremstyle{definition}
\newtheorem{proposition}[theorem]{Proposition}
\newtheorem{lemma}[theorem]{Lemma}
\newtheorem{definition}[theorem]{Definition}
\newtheorem{remark}[theorem]{Remark}
\newtheorem{corollary}[theorem]{Corollary}
\newtheorem{conjecture}[theorem]{Conjecture}

\newtheorem{example}[theorem]{Example}

\newtheorem{notation}[theorem]{Notation}
\newtheorem{convention}[theorem]{Convention}

\numberwithin{equation}{section}

\newcommand{\Z}{\mathbb Z}
\newcommand{\Q}{\mathbb Q}
\newcommand{\R}{\mathbb R} 
\newcommand{\C}{\mathbb C}

\renewcommand{\L}{\mathcal L}

\newcommand{\B}{\mathcal B}

\newcommand{\A}{\mathcal A}
\newcommand{\X}{\mathcal X}

\renewcommand{\P}{\mathcal P}

\DeclareMathOperator{\GL}{GL}
\DeclareMathOperator{\PSL}{PSL}

\DeclareMathOperator{\SL}{SL}

\DeclareMathOperator{\Real}{Re}
\DeclareMathOperator{\Imag}{Im}

\DeclareMathOperator{\ReIm}{ReIm}

\DeclareMathOperator{\id}{id}

\DeclareMathOperator{\Spec}{Spec}

\DeclareMathOperator{\Li}{Li}

\DeclareMathOperator{\Log}{Log}
\DeclareMathOperator{\Ker}{Ker}
\DeclareMathOperator{\Coker}{Coker}
\DeclareMathOperator{\Ext}{Ext}
\DeclareMathOperator{\Sym}{Sym}

\DeclareMathOperator{\ind}{ind}
\DeclareMathOperator{\Gr}{Gr}

\DeclareMathOperator{\cut}{cut}
\DeclareMathOperator{\odd}{odd}
\DeclareMathOperator{\even}{even}

\DeclareMathOperator{\sgn}{sgn}
\DeclareMathOperator{\Alt}{Alt}

\DeclareMathOperator{\AssGr}{gr}

\DeclareMathOperator{\Coeff}{Coeff}

\DeclareMathOperator{\symb}{symb}

\usepackage{ifpdf}
\ifpdf
\fi

\usepackage[left=2.8cm,right=2.8cm,bottom=3.1cm,top=3cm]{geometry}

\title[Holomorphic polylogarithms]{Holomorphic polylogarithms and Bloch complexes}
\author{Christian K. Zickert}
\address{University of Maryland \\
         Department of Mathematics \\
         College Park, MD 20742-4015, USA \newline
         {\tt \url{http://www.math.umd.edu/~zickert}}}
\email{zickert@math.umd.edu}

\thanks{C.~Z.~was supported in part by DMS-1711405 \\
\newline
2010 {\em Mathematics Classification.} Primary 11G55. %polylogs and K-theory
Secondary 13F60, 58J28, 19E15. %cluster algebras, Chern-Simons and Eta, alg cycles and motivic cohomology
}

\date{}
\begin{document}

\begin{abstract}
For an integer $n\geq 2$ we define a polylogarithm $\widehat\L_n$, which is a holomorphic function on the universal abelian cover of $\C\setminus\{0,1\}$ defined modulo $(2\pi i)^n/(n-1)!$.
We use the formal properties of its functional relations to define groups $\widehat\B_k(\widehat F)$ lifting Goncharov's Bloch groups $\B_k(F)$ of a field $F$, and show that they fit into a complex $\widehat\Gamma(F,n)$ lifting Goncharov's Bloch complex $\Gamma(F,n)$. When $F=\C$ we show that the imaginary part (when $n$ is even) or real part (when $n$ is odd) of $\widehat\L_n$ agrees with a real single valued polylogarithm $\L_n$ on the group $H^1(\widehat\Gamma(\C,n))$. When $n=2$, this group is Neumann's extended Bloch group. Goncharov's complex conjecturally computes the rational motivic cohomology of $F$, and one may speculate whether the extended complex computes the integral motivic cohomology. Finally, we use $\widehat\L_3$ to construct a lift of Goncharov's map $H_5(\SL(3,\C))\to\R$ to a complex valued map whose real part agrees with that of Goncharov. The lift makes use of the cluster ensemble structure on the Grassmannian $\Gr(3,6)$. 
\end{abstract}
\maketitle

\section{Introduction}
For a natural number $n$, the polylogarithm of weight $n$ is defined by the power series
\begin{equation}\label{eq:LiDef}\Li_n(z)=\sum_{k=1}^\infty z^k/k^n,\qquad \vert z\vert\leq 1.
\end{equation}
It extends holomorphically to $\C\setminus(1,\infty)$, but is multivalued on $\C$ with branch points at $0$ and $1$. There are several real single valued analogues of the polylogarithm (see~\cite{ZagierPolylogs} for definitions and basic properties). We shall only consider
\begin{equation}\label{eq:LnGonDef}
\L_n(z)=\mathfrak R_n\left (\sum_{r=0}^{n-1}\frac{2^rB_r}{r!}\Li_{n-r}(z)(\log\vert z\vert)^r\right ),
\end{equation}
where $\mathfrak R_n(x)$ denotes the real part of $x$ when $n$ is odd and the imaginary part when $n$ is even, and $B_0=1$, $B_1=-1/2$, $B_2=1/6$, $B_3=0$, $B_4=-1/30$, etc., are the Bernoulli numbers. The functions $\L_n(z)$ are continuous on $\C P^1=\C\cup\{\infty\}$, and
\begin{equation}
\L_2(z)=\Imag(\Li_2(z))+\Imag(\log(1-z))\log(\vert z\vert)
\end{equation}
is the Bloch-Wigner dilogarithm.
%For $n=2$, $3$ and $4$, we have
%\begin{equation}
%\begin{aligned}
%\L_2(z)&=\Imag(\Li_2(z))+\Imag(\log(1-z))\log(\vert z\vert)\\
%\L_3(z)&=\Real(\Li_3(z))-\Real(\Li_2(z))\log(\vert z\vert)-\frac{1}{3}\Real(\log(1-z))\log(\vert z\vert)^2,\\
%\L_4(z)&=\Imag(\Li_4(z))-\Imag(\Li_3(z))\log(\vert z\vert)+\frac{1}{3}\Imag(\Li_2(z))\log(\vert z\vert)^2.\\
%\end{aligned}
%\end{equation}
%In particular, $\L_2$ is the Bloch-Wigner dilogarithm.
\begin{notation} For a field $F$, $F^*=F\setminus\{0\}$ denotes the unit group and $P^1_F=F\cup\{\infty\}$. For an abelian group $A$, $A_\Q$ denotes $A\otimes\Q$. All tensor products and exterior powers are over $\Z$ unless otherwise specified. For a set $X$, $\Z[X]$ denotes the free abelian group generated by $X$, and for $x\in X$ the corresponding generator in $\Z[X]$ is denoted $[x]$. 

\end{notation}

For a field $F$, Goncharov~\cite{GoncharovConfigurations} has defined complexes $\Gamma(F,n)$ of the form% (tensor products and exterior powers are over $\Z$ and $F^*$ is the unit group of $F$)
\begin{equation}\label{eq:GammaComplex}
\cxymatrix{@C=2em{\B_n(F)\ar[r]^-{\delta}&\cdots\ar[r]^-{\delta}&\B_{n-k}(F)\otimes\wedge^k(F^*)\ar[r]^-{\delta}&\cdots\ar[r]^-{\delta}&\B_2(F)\otimes\wedge^{n-2}(F^*)\ar[r]^-{\delta}&\wedge^n(F^*).}}
\end{equation}
Each group $\B_k(F)$ is the quotient of $\Z[P^1_F]$ by a subgroup $R_k(F)$, which is defined inductively, and may be thought of as formal functional relations for $\L_k$. For example, for any $x,y\in F$ we have an element  (where $\frac{0}{0}=1$, $\frac{1}{0}=\infty$, etc.)
\begin{equation}\label{eq:FiveRelIntro}
[x]-[y]+[\frac{y}{x}]-[\frac{1-x^{-1}}{1-y^{-1}}]+[\frac{1-x}{1-y}]\in R_2(F)
\end{equation}
which corresponds to the well known functional relation
\begin{equation}\label{eq:FiveRelD}
\L_2(x)-\L_2(y)+\L_2(\frac{y}{x})-\L_2(\frac{1-x^{-1}}{1-y^{-1}})+\L_2(\frac{1-x}{1-y})=0,\qquad x,y\in \C.
\end{equation}

\begin{conjecture}[{Goncharov~\cite[Conj.~1.20]{GoncharovConfigurations}}]\label{conj:GonConj0} For $n=2$ and $n=3$, we have
\begin{enumerate}[label=\alph*)]
\item\label{R2Conj} $R_2(F)$ is generated by the five term relations~\eqref{eq:FiveRelIntro}.
\item\label{R3Conj} $R_3(F)$ is generated by the elementary relations $[x]-[x^{-1}]$, and $[x]+[1-x]+[1-x^{-1}]-[1]$ together with a 3-variable relation of the form $R_3(x,y,z)-3[1]$, where $R_3(x,y,z)$ is an explict element with 22 terms (see also~\cite{ZagierPolylogs}).
\end{enumerate}
\end{conjecture}
\begin{remark}For $n>3$ few elements in $R_n(F)$ are known other than $[x]+(-1)^n[x^{-1}]$ which is in $R_n(F)$ for any $n$. Gangl~\cite{Gangl931} has constructed a 931 term relation in $R_4(F)$; see also \cite{GoncharovRudenko} for an approach to describing $R_4(F)$ via cluster algebras. A discussion of Conjecture~\ref{conj:GonConj0}\ref{R2Conj} is given in~\cite{deJeuDilogRels}.
\end{remark}

\begin{conjecture}[{Goncharov~\cite[Conj.~2.1]{GoncharovMotivicGalois}}]\label{conj:GonConj1}
There are rational isomorphisms
\begin{equation}\label{eq:GoncharovConj}
H^i(\Gamma(F,n))_\Q\cong K^{(n)}_{2n-i}(F)_\Q,\qquad i=1,\dots, n
\end{equation}
where $K^{(p)}_q(F)=\AssGr_\gamma^p K_q(F)$ are the associated graded groups for the $\gamma$-filtration on $K_q(F)$.
\end{conjecture}
Goncharov further speculates (see~\cite{GoncharovRegulators}) that when $F=\C$, the map $\L_n$ (defined on $\B_n(\C)$ by linear extension) agrees with the Borel regulator map $b_n$, i.e.~that there is a commutative diagram
\begin{equation}\label{eq:RegnLn}
\cxymatrix{{{K^{(n)}_{2n-1}(\C)_\Q}\ar^-{\cong}[rr]\ar_-{b_n}[dr]&&H^1(\Gamma(\C,n))_\Q\ar^-{\L_n}[dl]\\&\R&}}
\end{equation}
\begin{remark}\label{KAndChow}
The groups $K^{(p)}_{q}(F)$ are rationally isomorphic to Blochs higher Chow groups $CH^p(F,q)$. In fact, one has $K^{(p)}_{q}(F)[\frac{1}{(q-1)!}]\cong CH^p(F,q)[\frac{1}{(q-1)!}]$~\cite{Levine}.
\end{remark}
\subsection{Motivic cohomology}
For a smooth scheme $X$ over a field $F$, Voevodsky has defined motivic cohomology groups $H^i_{\mathcal M}(X,\Z(n))$ satisfying (see \cite{Voevodsky,LectureNotesMotivic}) 
\begin{equation}
H^i_{\mathcal M}(X,\Z(n))\cong CH^n(X,2n-i),\qquad H^n(F,\Z(n))=K_n^M(F),
\end{equation}
where $K_n^M(F)$ is Milnor $K$-theory. By Remark~\ref{KAndChow} we reformulate Conjecture~\ref{conj:GonConj1}:
\begin{conjecture}\label{conj:GonConjMotivic}
\begin{equation}
H^i(\Gamma(F,n))_\Q\cong H^i_{\mathcal M}(F,\Z(n))_\Q.
\end{equation}
\end{conjecture}

\begin{example} When $n=2$ the complex $\Gamma(F,2)$ takes the form
\begin{equation}\label{eq:BlochSuslin}
\delta\colon\B_2(F)\to\wedge^2(F^*),\qquad [x]\mapsto x\wedge (1-x).
\end{equation}
Assuming Conjecture~\ref{conj:GonConj0}\ref{R2Conj}, the kernel of $\delta$ is (up to 6 torsion) the classical Bloch group $\B(F)$, and the cokernel is $K_2^M(F)$ by Matsumoto's theorem.
We also have (see~\cite{SuslinAlgKFields})
\begin{equation}\label{eq:MotivicGroupsnEq2}
H^1_{\mathcal M}(F,\Z(2))\cong CH^2(F,3)\cong K_3^{\ind}(F),\qquad H^2_{\mathcal M}(F,\Z(2))\cong K_2^M(F).
\end{equation}
Since $K_3^{\ind}(F)$ is an extension of $\B(F)$ by a non-trivial torsion group~\cite{Suslin}, it follows that Conjecture~\ref{conj:GonConj0}\ref{R2Conj} implies Conjecture~\ref{conj:GonConjMotivic} when $n=2$, and that Conjecture~\ref{conj:GonConjMotivic} is not true \emph{integrally}.
\end{example}

For a smooth scheme $X$ over $F$, Bloch~\cite{BlochAlgCycles} defined cycle maps from $CH^n(X,2n-i)$ to the Deligne cohomology group $H^i_{\mathcal D}(X,\Z(n))$ (see e.g.~\cite{EsnaultViehweg} for definition and basic properties of Deligne cohomology). Since $H^1_{\mathcal D}(\Spec(\C),\Z(n))=\C/(2\pi i)^n\Z$ it follows that there is a map
\begin{equation}\label{eq:cyclemap}
B_{n}\colon H^1_{\mathcal M}(\C;\Z(n))\to\C/(2\pi i)^n\Z.
\end{equation}

\begin{remark} In addition to the cycle maps $B_n$, one also has maps $K_{2n-1}^{(n)}(\C)\to \C/(2\pi i)^n\Z$ defined as the composition of the Hurevicz map $K_{2n-1}(\C)\to H_n(\GL(\C))$ with the Cheeger-Chern-Simons class $\widehat c_n\colon H_n(\GL(\C))\to\C/(2\pi i)^n\Z$ (see e.g.~\cite{CheegerSimons}). The two maps probably correspond, but the author is not aware of any proof of this.
\end{remark}

\subsection{Motivating goals}
A motivating goal for our work is to obtain a variant of Conjecture~\ref{conj:GonConjMotivic}, which holds integrally and a variant of~\eqref{eq:RegnLn} involving $B_n$ instead of $b_n$. In other words, we want a complex $\widehat\Gamma(F,n)$ such that
\begin{enumerate}
\item\label{Goal1} $H^i(\widehat\Gamma(F,n))$ is \emph{integrally} isomorphic to the motivic cohomology group $H^i_{\mathcal M}(F;\Z(n))$.
\item\label{Goal2} When $F=\C$, the map $B_{n}$ is given by an explicit polylogarithm on $H^1(\widehat\Gamma(\C,n))$.
\end{enumerate}

\subsection{Neumann's extended Bloch group} When $n=2$ and $F=\C$, one may interpret work of Neumann~\cite{Neumann} (after modifications~\cite{GoetteZickert,ZickertAlgK}; see Remark~\ref{rm:NeumannDiscuss} for a discussion) as an achievement of our goals. Neumann~\cite{Neumann} considered the universal abelian cover (also considered by Zagier~\cite{ZagierDilogarithm})  
\begin{equation}\label{eq:Chat}
r\colon\widehat\C=\left\{(u,v)\in\C^2\bigm\vert e^u+e^v=1\right\}\to\C\setminus\{0,1\},\qquad (u,v)\mapsto e^u
\end{equation}
of $\C\setminus\{0,1\}$ and defined an explicit dilogarithm function $R\colon\widehat\C\to\C/4\pi^2\Z$ satisfying five term relations lifting the relations~\eqref{eq:FiveRelIntro} (when $x\neq y\in\C\setminus\{0,1\})$. 
Letting $\widehat\P(\C)$ be the free abelian group on $\widehat\C$ modulo the lifted five term relations, Neumann defined a map (c.f.~\eqref{eq:BlochSuslin})
\begin{equation}\label{eq:NeumannComplex}
\widehat\nu\colon\widehat\P(\C)\to\wedge^2(\C),\qquad [(u,v)]\mapsto u\wedge v,
\end{equation}
and showed (see~\cite{GoetteZickert}) that there are isomorphisms
\begin{equation}
\Ker(\widehat\nu)\cong H_3(\SL(2,\C)),\qquad \Coker(\widehat\nu)\cong K_2^M(\C),
\end{equation}
and that the map $\Ker(\widehat\nu)\to\C/4\pi^2\Z$ induced by $R$ agrees with the second Cheeger-Chern-Simons class $\widehat c_2\colon H_3(\SL(2,\C))\to\C/4\pi^2\Z$.
 Since $H_3(\SL(2,\C))$ is isomorphic to $K_3^{\ind}(\C)$~\cite{Sah} it follows from~\eqref{eq:MotivicGroupsnEq2} that~\eqref{eq:NeumannComplex} may be viewed as a model for $\widehat\Gamma(\C,2)$. The group $\Ker(\widehat\nu)$ is called the \emph{extended Bloch group} and is denoted $\widehat\B(\C)$.

\subsubsection{Arbitrary fields}\label{sec:ArbitraryFieldsIntro} Although Neumann's work relies on analytic continuation it was generalized to arbitrary fields by Zickert~\cite{ZickertAlgK}. Given a torsion free $\Z$-extension $E$ of $F^*$, Zickert defined a group $\widehat\P_E(F)$ and a map
\begin{equation}\label{eq:NeumannComplexArbitraryF}
\widehat\nu\colon \widehat\P_E(F)\to\wedge^2(E)
\end{equation}
with cokernel $K_2^M(F)$. When $F=\C$ and $E=\C$ is the extension of $\C^*$ given by the exponential map, \eqref{eq:NeumannComplexArbitraryF} agrees with Neumann's complex~\eqref{eq:NeumannComplex}. The corresponding extended Bloch group $\widehat\B_E(F)=\Ker(\widehat\nu)$ only depends on the class of $E$ in $\Ext(F^*,\Z)$, and if $F^*$ is free modulo torsion and has finitely many roots of unity, $\widehat\B_E(F)$ is independent of $E$ up to natural isomorphism. If so, we suppress $E$ from the notation. If, in addition, $F$ admits an embedding in $\C$, the extended Bloch group $\widehat\B(F)$ is isomorphic to $K_3^{\ind}(F)$. In particular, \eqref{eq:NeumannComplexArbitraryF} is a model for $\widehat\Gamma(F,2)$ whenever $F$ is a finite extension of $\Q$.

\subsubsection{Variants} Neumann defined two variants of his extended Bloch group. The other variant is defined using the disconnected cover
\begin{equation}\label{eq:ChatSignsIntro}
\widehat \C_{\pm}=\Big\{(u,v)\in\C^2\bigm\vert \epsilon_1e^u+\epsilon_2e^v=1\text{ for some } \epsilon_1,\epsilon_2\in\{-1,1\}\Big\}.
\end{equation}
On this cover, the lifted five term relations are only defined modulo $\pi^2\Z$, and the resulting extended Bloch group $\widehat\B(\C)_{\pm}$ is the quotient of $\widehat\B(\C)$ by a cyclic subgroup of order 4. The main advantage of $\widehat\B(\C)_{\pm}$ is that elements seem to arise more naturally from other contexts. For example, a cusped hyperbolic $3$-manifold $M$ with an ideal triangulation naturally and explicitly determines an element in $\widehat\B(\C)_{\pm}$. A variant for arbitrary fields was defined by Zickert~\cite{ZickertAlgK}.

\begin{remark}\label{rm:NeumannDiscuss} We stress that our notation differs from that of Neumann~\cite{Neumann}. For example, Neumann used $\widehat\B(\C)$ to denote the variant using $\widehat\C_{\pm}$. Also, Neumann's original $R$ was only defined modulo $\pi^2$ and defined on $H_3(\PSL(2,\C))$ instead of $H_3(\SL(2,\C))$.  Our notation largely follows Zickert~\cite{ZickertAlgK}.
\end{remark}

\subsection{Summary of results}
In Section~\ref{sec:DefBasics} we define a holomorphic function
\begin{equation}
\widehat\L_n\colon \widehat\C\to\C/\frac{(2\pi i)^n}{(n-1)!}\Z
\end{equation}
and state its fundamental properties. The differential of $\widehat\L_n$ is the holomorphic 1-form
\begin{equation}
d\widehat\L_n=(-1)^n\frac{n-1}{n!}u^{n-2}(udv-vdu)\in\Omega^1(\widehat\C).
\end{equation}
The $\widehat\L_n$ are also defined on the disconnected cover $\widehat\C_{\pm}$, where the ambiguity of definition is reduced by a power of $2$ depending on $n$ (see Theorem~\ref{thm:DefAndAmb}).
In Section~\ref{sec:RelsAndSymb} we show that certain functional relations for $\L_n$ give rise to analogous functional relations for $\widehat\L_n$, the idea being that the vanishing of the symbol map is equivalent to the vanishing of a certain symbolic 1-form corresponding to $d\widehat\L_n$. The $\widehat\L_n$ relations are sometimes more naturally defined on $\widehat\C_{\pm}$. 

In general, the relations depend on choices of log branches, but it turns out that there is a completely combinatorial (inductive) description of relations where the value of $\widehat\L_n$ is independent of these choices. This allows us to define relations over arbitary fields. As in Zickert~\cite{ZickertAlgK} this is done using a choice of torsion free $\Z$-extension $\pi\colon E\to F^*$ (if $F=\C$, we choose the exponential extension). Given this choice, we have a natural algebraic analogue
\begin{equation}
\widehat F=\Big\{(u,v)\in E\times E\bigm\vert \pi(u)+\pi(v)=1\Big\}
\end{equation}
of Neumann's cover $\widehat\C$. In Section~\ref{sec:LiftedBloch} we define subgroups $\widetilde R_k(F)$ of $\Z[\widehat F]$ such that the groups $\widehat\B_k(\widehat F)=\Z[\widehat F]/\widetilde R_k(F)$ fit in a chain complex $\widehat\Gamma(F,n)$ of the form

\begin{equation}\label{eq:LiftedGammaComplex}
\cxymatrix{@C=2em{{\widehat\B_n(\widehat F)}\ar[r]^-{\delta}&\cdots\ar[r]^-{\delta}&{\widehat\B_{n-k}(\widehat F)\otimes\wedge^k(E)}\ar[r]^-{\delta}&\cdots\ar[r]^-{\delta}&{\widehat\B_2(\widehat F)\otimes\wedge^{n-2}(E)}\ar[r]^-{\delta}&\wedge^n(E).}}
\end{equation}
The natural map $r\colon\widehat F\to F\setminus\{0,1\}$ taking $(u,v)$ to $\pi(u)$ induces a chain map $\widehat\Gamma(F,n)\to\Gamma(F,n)$. The fact that $\widetilde R_n(F)$ consists of formal functional relations for $\widehat\L_n$ is proved in Section~\ref{sec:RnEqLnhatRels}, and in Section~\ref{sec:Comparison} we show that  $\mathfrak R_n\circ\widehat\L_n=\L_n\circ r$ on $H^1(\widehat\Gamma(\C,n))$. We stress that this equality only holds in cohomology and not on $\widehat\C$ (see Theorem~\ref{thm:CompWithLnStatement}). This is because $\widehat\L_n$ is holomorphic, but $\L_n$ is not the real or imaginary part of a holomorphic function. Assuming that $\widetilde R_2(F)$ is generated by lifted five term relations, $\widehat\Gamma(F,2)$ is Zickert's complex~\eqref{eq:NeumannComplexArbitraryF}, so under this assumption $\widehat\Gamma(F,2)$ satisfies our motivational goal. We speculate that this holds more generally, i.e.~that $\widehat\Gamma(F,n)$ computes the motivic cohomology groups $H^i_{\mathcal M}(F,\Z(n))$ and that when $F=\C$ the cycle map~\eqref{eq:cyclemap} agrees with $(n-1)!\widehat\L_n$ under the isomorphism $H^1(\widehat\Gamma(\C,n))\cong H^1_{\mathcal M}(\C,\Z(n))$.

%\begin{equation}\label{eq:IntroLnhatandLn}
%\cxymatrix{{H^1(\widehat\Gamma(\C,n))\ar[r]^-{\widehat\L_n}\ar[d]^-r&\C/\frac{(2\pi i)^n}{(n-1)!}\ar[d]^-{\mathfrak R_n}\\H^1(\Gamma(\C,n))\ar[r]^-{\L_n}&\R}}
%\end{equation}

In Section~\ref{sec:LiftOfReg} we lift Goncharov's map $H_5(\SL(3,\C))\to\R$ to a complex valued map. Goncharov's map is obtained as a composition
\begin{equation}
H_5(\SL(3,\C))\to H^1(\Gamma(\C,3)_\Q)\overset{\L_3}{\to}\R,
\end{equation}
and our map is defined similarly using $H^1(\widehat\Gamma(\C,3))$ and $\widehat\L_3$ and is thus defined modulo $\frac{(2\pi i)^3}{2}\Z$. Our map is expressed in terms of $\X$-coordinates on the affine Grassmannian $\widetilde{\Gr}(3,6)$, and we speculate that twice our map equals the third Cheeger-Chern-Simons class $\widehat c_3$~\cite{CheegerSimons}. This would be a natural generalization of the fact that $\widehat\L_2$ on the extended Bloch group equals $\widehat c_2$.

\begin{remark} The exterior algebra $\wedge^*(A)$ of an abelian group $A$ is regarded as the quotient of the tensor algebra by the relations $a\otimes b+b\otimes a$. In Goncharov's definition of $\Gamma(F,n)$ he additionally assumes that $x\wedge x=x\wedge (-x)=0\in \wedge^2(F^*)$. In the definition of $\widehat\Gamma(F,n)$ we assume that $a\wedge a=0$ when $n>2$ (to ensure that $\delta^2=0$). Since the variants differ only by 2-torsion we shall for notational simplicity denote them all by $\wedge^*(A)$.
\end{remark}

\begin{remark}
We do not know the kernel and cokernel of the map $H^i(\widehat\Gamma(F,n))\to H^i(\Gamma(F,n))$, but we suspect that the cokernel is trivial and that the kernel is torsion. Assuming that $\widetilde R_2(F)$ is generated by lifted five term relations this follows from~\cite{ZickertAlgK} when $n=2$ and $F$ is a finite extension of $\Q$.
We also do not know how the choice of extension $E$ affects the groups $\widehat\B_n(\widehat F)$, but we expect that the results of Zickert for $n=2$ (see Section~\ref{sec:ArbitraryFieldsIntro}) hold for $n>2$ also.
\end{remark}

\begin{remark} We also consider a variant $\widehat\B_n(\widehat F)_{\pm}$ defined using an algebraic analogue $\widehat F_{\pm}$ of $\widehat\C_{\pm}$. The main difference is that elements are easier to produce. For example, one has the element $\alpha=[(u,v)]+(-1)^n[(-u,v-u)]$ in $\Z[\widehat F_{\pm}]$, which may be considered a lift of $\beta=[x]+(-1)^n[x^{-1}]$. The fact that $\L_n(\beta)=0$ when $F=\C$, whereas the order of $\widehat\L_n(\alpha)$ is related to the torsion in $H^1_{\mathcal M}(\Z,\Z(n))$ (Corollary~\ref{cor:TorsionKTheory}) provides additional support for the speculative relationship to motivic cohomology.
\end{remark}

%\begin{remark}
%Goncharov's complex $\Gamma(F,n)$ is only a chain complex up to 2-torsion (this is because $2x\wedge x$ is zero, but $x\wedge x$ may not be). The same is true for our complex $\widehat\Gamma(F,n)$ when $n>2$. If we replace each $\wedge^k(E)$  by its quotients by 2-torsion when $n>2$, we obtain a true chain complex.
%\end{remark}

\subsection*{Acknowledgment}
We thank Soren Galatius, Matthias Goerner, Zachary Greenberg, Dani Kaufman, Markus Spitzweck, and the anonymous referee for helpful comments. The work was supported in part by NSF grant DMS-1711405.

\section{Definition and basic properties of $\widehat\L_n$}\label{sec:DefBasics}
Fix an integer $n\geq 2$. Recall the space $\widehat\C_{\pm}$ defined in~\eqref{eq:ChatSignsIntro}. It has four components, which we denote $\widehat\C_{++}$, $\widehat\C_{-+}$, $\widehat\C_{+-}$ and $\widehat\C_{--}$ corresponding to the signs of $\epsilon_1$ and $\epsilon_2$. Note that $\widehat\C_{++}$ is the space $\widehat\C$ defined in~\eqref{eq:Chat}.
There is a holomorphic map
\begin{equation}\label{eq:CoveringMap}
r\colon\widehat\C_{\pm}\to\C\setminus\{0,1\},\qquad (u,v)\mapsto \epsilon_1 e^u\quad \text{ if }(u,v)\in\widehat\C_{\epsilon_1,\epsilon_2},
\end{equation}
which restricts to a $\Z\times\Z$ cover on each component. 
On $\widehat\C_{\pm}$ we introduce the holomorphic 1-form
\begin{equation}\label{eq:wnDef}
\omega_n=(-1)^n\frac{n-1}{n!}u^{n-2}\big(udv-vdu\big)\in\Omega^1(\widehat\C_{\pm}),
\end{equation}
which is closed since $\widehat\C_{\pm}$ is complex $1$-dimensional. Let $\nu_2$ denote the $2$-adic valuation and let
\begin{equation}\label{eq:KappaDef}
\kappa_n=\begin{cases}2^{2-n}&\text{if }n\text{ is even,}\\2^{3+\nu_2(n-1)-n}&\text{if }n\text{ is odd.}\end{cases}
\end{equation}

\begin{theorem}[Proof in Section~\ref{sec:DefLnSec}]\label{thm:IntegralPeriods} The form $w_n$ has periods in $\frac{(2\pi i)^n}{(n-1)!}\Z$ on $\widehat\C_{++}$ and $\widehat\C_{+-}$ and periods in $\kappa_n\frac{(2\pi i)^n}{(n-1)!}\Z$ on $\widehat\C_{-+}$ and $\widehat\C_{--}$.
\end{theorem}
%Note that on $\widehat\C_{-+}$, $\widehat\L_3$ is defined modulo $(2\pi i)^3$.
\subsection{Primitives for $\omega_n$}
Let $\Log$ denote the main branch of logarithm (argument in $(-\pi,\pi]$) and let $\Li_k$ denote the main branch of polylogarithm ($\Li_1(z)=-\Log(1-z)$, $\Li_k(z)=\int_{0}^1\frac{\Li_{k-1}(tz)}{t}dt$). We may uniquely write each element $(u,v)$ in $\widehat\C_{\pm}$ in the form 
\begin{equation}\label{eq:GoodRep}
\langle z;p,q\rangle_{\epsilon_1,\epsilon_2}:=(\Log(\epsilon_1 z)+2p\pi i,\Log(\epsilon_2(1-z))+2q\pi i).
\end{equation} 
For $z\in\C\setminus\{0,1\}$ and an integer $q$ let
\begin{equation}
\L i_k(z;q)=\Li_{k}(z)-\frac{2q\pi i}{(k-1)!}\Log(z)^{k-1}.
\end{equation}
\begin{theorem}[Proof in Section~\ref{sec:DefLnSec}]\label{thm:DefAndAmb} Let $(u,v)=\langle z;p,q\rangle_{\epsilon_1,\epsilon_2}$ as in~\eqref{eq:GoodRep}.
The function
\begin{equation}\label{eq:LnhatIntro}
\widehat\L_n(u,v)=\sum_{r=0}^{n-1}\frac{(-1)^r}{r!}\L i_{n-r}(z;q)u^r-\frac{(-1)^n}{n!}u^{n-1}v
\end{equation}
is holomorphic and well defined modulo $\frac{(2\pi i)^n}{(n-1)!}$ for $(u,v)\in\widehat\C_{++}$ or $\widehat\C_{+-}$ and modulo $\kappa_n\frac{(2\pi i)^n}{(n-1)!}$ for $(u,v)\in\widehat\C_{-+}$. It is a primitive for $\omega_n$, i.e.~$d\widehat\L_n=\omega_n$.
\end{theorem}
\begin{remark}
On $\widehat\C_{--}$ the function \eqref{eq:LnhatIntro} is only defined modulo $\frac{(\pi i)^n}{(n-1)!}$, and in order to obtain a primitive defined modulo $\kappa_n\frac{(2\pi i)^n}{(n-1)!}$, one must modify it by multiples of $\frac{(\pi i)^n}{(n-1)!}$. We refer to Section~\ref{sec:BasicProperties} for details.
\end{remark}

%\begin{remark} Neumann~\cite{Neumann} (see also~\cite{GoetteZickert,ZickertAlgK}) writes an element in $\widehat\C_{\pm}$ as 
%\begin{equation}
%[z;p,q]=(\Log(z)+p\pi i,\Log(1-z)+q\pi i)
%\end{equation}
%and considers the function
%\begin{equation}
%\mathcal R\colon\widehat\C_{\pm}\to\C/\pi^2\Z,\qquad [z;p,q]\mapsto \Li_2(z)+\frac{1}{2}(\Log(z)+p\pi i)(\Log(1-z)-q\pi i)-\frac{\pi^2}{6}.
%\end{equation}
%One easily checks that $\widehat\L_2$ equals $\mathcal R+\frac{\pi^2}{6}$ modulo $\pi^2$.
%\end{remark}
\begin{remark}
The map $\widehat\L_2$ equals $R+\frac{\pi^2}{6}$ modulo $\pi^2$, where $R$ is Neumann's polylogarithm~\cite{Neumann}.
\end{remark}

\subsection{Inversions and order 3 symmetries in low degree}\label{sec:InvSymResults}
It is well known that the polylogarithm $\L_n$ in~\eqref{eq:LnGonDef} satisfies the functional equations
\begin{equation}
\begin{gathered}
\L_n(z)+(-1)^n\L_n(z^{-1})=0,\\
\L_3(z)+\L_3(\frac{1}{1-z})+\L_3(1-z^{-1})=\zeta(3), \qquad \L_2(z)-\L_2(\frac{1}{1-z})=0. 
\end{gathered}
\end{equation}
Consider the holomorphic maps
\begin{equation}\label{eq:SigmahatTauhat}
\tau\colon\widehat\C_{\pm}\to\widehat\C_{\pm},\quad (u,v)\mapsto (-u,v-u),\qquad 
\sigma\colon\widehat\C_{\pm}\to\widehat\C_{\pm},\quad (u,v)\mapsto (-v,u-v).
\end{equation}
One easily checks that $\tau$ and $\sigma$ have order $2$ and $3$, respectively, and that they are lifts of the maps $\C\setminus\{0,1\}\to\C\setminus\{0,1\}$ given by $z\mapsto z^{-1}$ and $z\mapsto \frac{1}{1-z}$, respectively. We stress that $\tau$ and $\sigma$ are not defined on $\widehat\C$. An elementary calculation shows that ($*$ denotes pullback of forms)
\begin{equation}
\tau^*(\omega_n)=-(-1)^n\omega_n, \qquad \omega_3+\sigma^*\omega_3+\sigma^*(\sigma^*(\omega_3))=0, \qquad \sigma^*(\omega_2)=\omega_2.
\end{equation}
This implies that the functions
\begin{equation}
\begin{gathered}
\widehat\L_n(u,v)+(-1)^n\widehat\L_n(-u,v-u),\\\widehat\L_3(u,v)+\widehat\L_3(-v,u-v)+\widehat\L_3(v-u,-u),\qquad \widehat\L_2(u,v)-\widehat\L_2(-v,u-v),
\end{gathered}
\end{equation}
are locally constant.

\begin{proposition}\label{prop:tauhat} For any $(u,v)\in\C_{-+}$ we have
\begin{equation}\label{eq:Tauhat}
\widehat\L_n(u,v)+(-1)^n\widehat\L_n(-u,v-u) =(2^n-2)(\pi i)^n\frac{B_n}{n!}\in\C\big/\kappa_n\frac{(2\pi i)^n}{(n-1)!}\Z.
\end{equation}
In particular, it is zero if $n$ is odd.
\end{proposition}
\begin{proof} Since the left-hand side is constant on $\widehat\C_{-+}$ it is enough to consider $(u,v)=(0,\Log(2))\in\widehat\C_{-+}$. For this point $\tau(u,v)=(u,v)$, which proves the result for odd $n$.  Since $\Li_n(1)=\zeta(n)$ (this follows from~\eqref{eq:LiDef}) it follows from the formula $\Li_n(z)+\Li_n(-z)=2^{1-n}\Li_n(z^2)$~\cite[p.~29]{PolylogBook} that $\Li_n(-1)=-(1-2^{1-n})\zeta(n)$. When $n$ is even, $\zeta(n)=(-1)^{\frac{n}{2}+1}\frac{B_n(2\pi)^n}{2n!}$, so we have
\begin{equation}
2\widehat\L_n(0,\log(2))=2\Li_{n}(-1)=-2(1-2^{1-n})\zeta(n)=(2^n-2)(\pi i)^n\frac{B_n}{n!}.
\end{equation} 
This concludes the proof.
\end{proof}

\begin{corollary}\label{cor:TorsionKTheory} When $n$ is even, the order of~\eqref{eq:Tauhat} is the order of the torsion in $H^1_{\mathcal M}(\Z,\Z(n))$.
\end{corollary}
\begin{proof}
By the Staudt-Clausen formula for the denominator of even Bernoulli numbers, we see that the order of \eqref{eq:Tauhat} is equal to the denominator of $\frac{B_n}{2n}$. The $K$-groups $K_{2n-1}(\Z)$ are given in terms of integers $w_n(\Q)$ in~\cite[Thm.~1]{Weibel}, and it follows from the exact sequences in \cite[Thm~14.10, Rm.~14.11]{LevineSchemes}, relating $K_{2n-1}(\mathcal O_S)$ to $H^1_{\mathcal M}(\mathcal O_S,\Z(n))$ for number rings $\mathcal O_S$, that the torsion in $H^1_{\mathcal M}(\Z,\Z(n))$ has order $w_n(\Q)$ for any $n$. The result now follows from the fact $w_n(\Q)$ is the denominator of $\frac{B_n}{2n}$ when $n$ is even. See also \cite{Spitzweck} for a table of $H^p_{\mathcal M}(\Z,\Z(q))$.
\end{proof}

\begin{remark}\label{rm:OtherComp}
Proposition~\ref{prop:tauhat} also holds for $(u,v)\in\widehat\C_{--}$. For the other 2 components the right-hand side is $-(2\pi i)^n\frac{B_n}{n!}$ modulo $\frac{(2\pi i)^n}{(n-1)!}$. The order of this is half the order of $H^1_{\mathcal M}(\Z,\Z(n))$.
\end{remark}

\begin{lemma}\label{lemma:sigmahat3} For $(u,v)$ in $\widehat\C_{++}$, $\widehat \C_{+-}$ or $\widehat\C_{-+}$ we have
\begin{equation}\label{eq:Sigmahat}
\widehat\L_3(u,v)+\widehat\L_3(-v,u-v)+\widehat\L_3(v-u,-u)=\zeta(3) \mod 4\pi^3 i.
\end{equation}
\end{lemma}
\begin{proof}
It is enough to check this for $(u,v)\in\widehat\C_{++}$. If $(u,v)$ is a lift of $-1$, then $(-v,u-v)$ and $(v-u,-u)$ are lifts of $\frac{1}{2}$ and $2$, respectively. The result now follows from the formulas
\begin{equation}
\begin{gathered}
\Li_3(-1)=-\frac{3}{4}\zeta(3),\qquad \Li_3(\frac{1}{2})=\frac{7}{8}\zeta(3)-\frac{\pi^2}{12}\log(2)+\frac{1}{6}\log(2)^3,\\\Li_3(2)=\frac{7}{8}\zeta(3)+\frac{\pi^2}{4}\log(2)^2-\frac{\pi}{2}i\log(2)^2.
\end{gathered}
\end{equation}
which can be found in~\cite[A.2.6]{LewinPolyLogs}. We leave the details of the computation to the reader.
\end{proof}
%\begin{remark}
%One has $\widehat\L_3(u,v)+\widehat\L_3(-v,u-v)+\widehat\L_3(v-u,-u)=\zeta(3)-\frac{3}{2}\pi^3 i$  modulo $8\pi^3 i$ if $(u,v)\in\widehat\C_{--}$.
%\end{remark}
\begin{lemma}\label{lemma:sigmahat2}  We have
\begin{equation}\widehat\L_2(u,v)-\widehat\L_2(-v,u-v)=-\frac{\pi^2}{6} \mod \frac{\pi^2}{2}.
\end{equation}
\end{lemma}
\begin{proof} This is proved using elementary properties of the dilogarithm. We omit the details.
\end{proof}
%\begin{equation}\label{eq:Sigmahat}
%\widehat\L_3(u,v)+\widehat\L_3(-v,u-v)+\widehat\L_3(v-u,-u)=\begin{cases}\zeta(3)-\frac{3}{2}\pi^3i&\text{if }(u,v)\in\widehat\C_{--}\\\zeta(3)&\text{otherwise}.\end{cases}
%\end{equation}

\subsection{Relationship between $\widehat\L_n$ and $\L_n$}
For $(u,v)\in\C^2$, let 
\begin{equation}
\det(u\wedge v)=\Real(u)\Imag(v)-\Imag(u)\Real(v).
\end{equation} 
The result below relates $\widehat\L_n$ to $\L_n$.
It is a generalization of~\cite[Prop.~4.6]{DupontZickert} for $n=2$.
\begin{theorem}[Proof in Section~\ref{sec:Comparison}]\label{thm:CompWithLnStatement}  There exist rational numbers $c_{i,j}$ and $d_{i,j}$ such that
\begin{equation}\label{eq:CompWithLnStatement}
\begin{aligned}
\mathfrak R_n(\widehat\L_n(u,v))-\L_n(r(u,v)))&=\sum_{s=1}^{n-2}\bigg(\mathfrak R_{n-s}(\widehat\L_{n-s}(u,v))\sum_{i=0}^s c_{i,s-i}\Real(u)^i\Imag(u)^{s-i}\bigg)+\\&\det(u\wedge v)\sum_{i=0}^{n-2}d_{i,n-2-i}\Real(u)^i\Imag(u)^{n-2-i}.
\end{aligned}
\end{equation}
\end{theorem}
Explicit formulas for $c_{i,j}$ and $d_{i,j}$ are given in Section~\ref{sec:Comparison}. For example, we have
\begin{equation}
\begin{aligned}
\Imag(\widehat\L_2(u,v))-\L_2(r(u,v))=&-\frac{1}{2}\det(u\wedge v)\\
\Real(\widehat\L_3(u,v))-\L_3(r(u,v))=&\Imag(\widehat\L_2(u,v))\Imag(u)+\frac{1}{6}\det(u\wedge v)\\
\Imag(\widehat\L_4(u,v))-\L_4(r(u,v))=&-\Real(\widehat\L_3(u,v))\Imag(u)+\frac{1}{6}(\Real(u)^2+3\Imag(u)^2)\Imag(\widehat\L_2(u,v))+\\&\frac{1}{24}\det(u\wedge v)(\Real(u)^2+\Imag(u)^2).
\end{aligned}
\end{equation}

\begin{remark} This result is the key to proving that $\mathfrak R_n\circ\widehat\L_n$ agrees with $\L_n\circ r$ on $H^1(\widehat\Gamma(\C,n))$.
\end{remark}

%\subsection{Changing the lift} The result below shows how different lifts affect the value of $\widehat\L_n$.
%
%\begin{proposition}[Proof in Section~\ref{}]\label{prop:LhatDifference}  Let $\bar k=k\pi i$ and $\bar l= l\pi i$ denote multiples of $\pi i$.
%\begin{multline}\label{eq:LhatDifference}
%\widehat\L_n(u+\bar k,v+\bar l)-\widehat\L_n(u,v)=\sum_{r=1}^{n-2}(-1)^r\frac{\bar k^r}{r!}\widehat\L_{n-r}(u,v)+\frac{(-1)^n}{n!}A(u,v;\bar k,\bar l)+\\\frac{(-1)^n}{n!}\sum_{r=0}^{n-3}A_r(u,v;\bar k,\bar l)-
%(-1)^n\frac{\bar k^{n-1}\bar l}{n!},
%\end{multline}
%modulo $\frac{(\pi i)^n}{(n-1)!}$ where
%\begin{equation}
%A(u,v;\bar k,\bar l)=(\bar kv-\bar lu)(u+\bar k)^{n-2},\qquad A_r(u,v;\bar k,\bar l)=(\bar kv-\bar lu)\binom{n-2}{r+1}u^r\bar k^{n-2-r}.
%\end{equation}
%Moreover, if $(u,v)\in\widehat\C_{++}$ and $\bar k$ and $\bar l$ are even multiples of $\pi i$, then~\eqref{eq:LhatDifference} holds modulo $\frac{(2\pi i)^n}{(n-1)!}$.
%\end{proposition}
%
%\begin{remark}
%The reason for writing the right-hand side in this particular way will become clear later.
%\end{remark}

\subsection{Limiting behavior} We have $\L_n(0)=0$, but since $0$ has no lift, there is no obvious analogue for $\widehat\L_n$.
\begin{lemma}\label{lemma:Cauchy}
Let $z_k\in\C\setminus\{0,1\}$ be a sequence with $z_k\rightarrow 0$, and let $\{p_k\}\subset\Z$ be a bounded sequence. Letting $k\to\infty$ we have
\begin{equation}\label{eq:LimitSeq}
\widehat\L_n(\Log(z_k)+p_k\pi i,\Log(1-z_k))\rightarrow 0, \qquad \widehat\L_n(\Log(1-z_k),\Log(z_k)+2p_k\pi i)\rightarrow \zeta(n)
\end{equation}
\end{lemma}
\begin{proof}
By L'Hospital's rule, the sequence $\Li_s(z_k)\Log(z_k)^r$ tends to $0$ for any positive integers $r$ and $s$. This proves the first limit. The second limit follows from the fact that $\Li_n(1)=\zeta(n)$.
\end{proof}

\begin{remark} The derivative only determines $\widehat\L_n$ up to an integration constant. Lemma~\ref{lemma:Cauchy} determines the constant on $\widehat\C_{++}$ and $\widehat\C_{-+}$. The constant is then specified on $\widehat\C_{+-}$ by Remark~\ref{rm:OtherComp}, which also specifies the constant modulo 2-torsion on $\widehat\C_{--}$ when $n$ is even.
\end{remark}

\subsection{The polylogarithm formula for $\widehat\L_n$}\label{sec:BasicProperties}
We now prove that the 1-forms $\omega_n$ have integral periods and that our formula for $\widehat\L_n$ is a primitive (Theorems~\ref{thm:IntegralPeriods} and \ref{thm:DefAndAmb}). This is straightforward, but fairly technical, so we present detailed arguments.

We first give concrete models for each of the four components of $\widehat\C_{\pm}$ following~\cite{Neumann,GoetteZickert}. For signs $\epsilon_1$ and $\epsilon_2$ (regarded when convenient as elements of $\{-1,1\}$), let
\begin{equation}
\C^{\cut}_{\epsilon_1\epsilon_2}=\C\setminus\left\{z\in\R\mid \epsilon_1z\leq 0,\enspace \epsilon_2(1-z)\leq 0\right\}.
\end{equation}
Note that $\C^{\cut}_{--}$ is disconnected.
Let 
\begin{equation}
\overline\C^{\cut}_{\epsilon_1\epsilon_2}=\C^{\cut}_{\epsilon_1\epsilon_2}\cup\left\{z\pm0i\mid z\in\R,\enspace \epsilon_1z<0,\enspace \epsilon_2(1-z)<0\right\}.
\end{equation}
The functions $\Li_k$ and $\Log$ extend continuously to $\overline\C^{\cut}_{\epsilon_1\epsilon_2}$. Define $\widehat\C_{\epsilon_1\epsilon_2}$ to be the Riemann surface obtained from 
$\overline\C^{\cut}_{\epsilon_1\epsilon_2}\times \Z^2$ as the quotient by the relations
\begin{equation}
\begin{aligned}
&(z+0i,p,q)\sim(z-0i,p+\epsilon_1,q)&\text{if }\epsilon_1z<0,\quad \epsilon_2(1-z)>0&\\
&(z+0i,p,q)\sim(z-0i,p,q-\epsilon_2)&\text{if }\epsilon_1z>0,\quad \epsilon_2(1-z)<0&\\
&(z+0i,p,q)\sim(z-0i,p+\epsilon_1,q-\epsilon_2)&\text{if }\epsilon_1z<0,\quad \epsilon_2(1-z)<0&.
\end{aligned}
\end{equation}
An equivalence class is denoted by $\langle z;p,q\rangle_{\epsilon_1\epsilon_2}$. The map
\begin{equation}
\langle z;p,q\rangle_{\epsilon_1\epsilon_2}\mapsto (\Log(\epsilon_1z)+2p\pi i,\Log(\epsilon_2(1-z))+2q\pi i)
\end{equation}
identifies $\widehat\C_{\epsilon_1\epsilon_2}$ with the appropriate component of $\widehat\C_{\pm}$.
%\begin{equation}
%\widehat L_n(\langle z;p,q\rangle_{\epsilon_1\epsilon_2})=\sum_{r=0}^{n-1}\frac{(-1)^r}{r!}\L i_{n-r}(z,q)u^r-\frac{(-1)^n}{n!}vu^{n-1}, 
%\end{equation}
\subsubsection{Definition of $\widehat\L_n$}\label{sec:DefLnSec}
We begin with the definition of a map $\widehat L_n$, which agrees with $\widehat\L_n$ except on $\widehat\C_{--}$. Consider the map
\begin{equation}
\widehat L_n\colon\overline\C^{\cut}_{\epsilon_1\epsilon_2}\times \Z^2\to \C
\end{equation}
taking $\langle z;p,q\rangle_{\epsilon_1\epsilon_2}$ to
\begin{equation}
\sum_{r=0}^{n-1}\frac{(-1)^r}{r!}\L i_{n-r}(z;q)\Log(\epsilon_1z;p)^r-\frac{(-1)^n}{n!}\Log(\epsilon_1z;p)^{n-1}\Log(\epsilon_2(1-z);q),
\end{equation}
where
\begin{equation}
\L i_k(z;q)=\Li_{k}(z)-\frac{2q\pi i}{(k-1)!}\Log(z)^{k-1},\qquad \Log(z;p)=\Log(z)+2p\pi i.
\end{equation}
We wish to show that $\widehat L_n$ descends to a holomorphic function on $\widehat\C_{\epsilon_1\epsilon_2}$. For $z\in\R\setminus\{0,1\}$, let $z_\pm=z\pm0i$, and let
\begin{equation}
\Delta_{\epsilon_1\epsilon_2}(z,p,q)=\frac{(n-1)!}{(2\pi i)^n}
\begin{cases}
\widehat L_n(z_+,p,q)-\widehat L_n(z_-,p+\epsilon_1,q)&\text{if }\epsilon_1z<0,\quad \epsilon_2(1-z)>0\\
\widehat L_n(z_+,p,q)-\widehat L_n(z_-,p,q-\epsilon_2)&\text{if }\epsilon_1z>0,\quad \epsilon_2(1-z)<0\\
\widehat L_n(z_+,p,q)-\widehat L_n(z_-,p+\epsilon_1,q-\epsilon_2)&\text{if }\epsilon_1z<0,\quad \epsilon_2(1-z)<0\\
0&\text{if }\epsilon_1z>0,\quad \epsilon_2(1-z)>0.\end{cases}
\end{equation}

Clearly, $\Delta_{\epsilon_1\epsilon_2}(z,p,q)$ only depends on the interval $I$ (either $(-\infty,0)$, $(0,1)$, or $(1,\infty)$) where $z$ belongs. We denote it by $\Delta^I_{\epsilon_1\epsilon_2}(p,q)$ accordingly. Let
\begin{equation}
\delta(p,n)=(-1)^n((p-1)^{n-1}-p^{n-1}).
\end{equation}

\begin{lemma}\label{lemma:DeltaFormulas} We have
\begin{equation}
\begin{aligned}
\Delta_{++}^{(-\infty,0)}(p,q)&=q\delta(p+1,n),&\Delta_{++}^{(0,1)}(p,q)&=\Delta_{++}^{(1,\infty)}(p,q)=0\\
\Delta_{-+}^{(-\infty,0)}(p,q)&=q\delta(p+\frac{1}{2},n),&\Delta_{-+}^{(0,1)}(p,q)&=\Delta_{-+}^{(1,\infty)}(p,q)=0\\
\Delta_{+-}^{(-\infty,0)}(p,q)&=(-p-1)^{n-1}+q\delta(p+1,n),&\Delta_{+-}^{(0,1)}(p,q)&=\Delta_{+-}^{(1,\infty)}(p,q)=(-p)^{n-1}\\
\Delta_{--}^{(-\infty,0)}(p,q)&=(-p-\frac{1}{2})^{n-1}+q\delta(p+\frac{1}{2},n),&\Delta_{--}^{(0,1)}(p,q)&=\Delta_{--}^{(1,\infty)}(p,q)=(\frac{1}{2}-p)^{n-1}
\end{aligned}
\end{equation}
\end{lemma}

\begin{proof}
Suppose $z<0$. Then $\Li_k(z_+)=\Li_k(z_-)$ and we have
\begin{multline}
\widehat L_n(\langle z_+;p,q\rangle_{++})-\widehat L_n(\langle z_-;p+1,q\rangle_{++})=\\
-\frac{2q\pi i}{(n-1)!}\Bigg(\sum_{r=0}^{n-1}\Big(\textstyle{\binom{n-1}{r}}(\Log(z_+))^{n-r-1}(-\Log(z_+;p)^r\Big)-\\\sum_{r=0}^{n-1}\Big(\textstyle{\binom{n-1}{r}}\Log(z_-)^{n-r-1}(-\Log(z_-;p+1)^r\Big)\Bigg)=\\
-\frac{2q\pi i}{(n-1)!}\big((-2p\pi i)^{n-1}-(-2(p+1)\pi i)^{n-1}\big)=\frac{(2\pi i)^n}{(n-1)!}q\delta(p+1,n).
\end{multline}
This proves the first equality. Some of the other ones make use of the identity
\begin{equation}
\Li_n(z_+)-\Li_n(z_-)=\frac{2\pi i\Log(z)^{n-1}}{(n-1)!},\quad z\in(1,\infty)
\end{equation}
but are otherwise similar. We leave their verification to the reader.
\end{proof}

\begin{corollary}\label{cor:Amb}
$\widehat L_n$ is holomorphic on $\widehat\C_{\pm}$ and defined modulo $\frac{(2\pi i)^n}{(n-1)!}$ on $\widehat\C_{++}$ and $\widehat\C_{+-}$, modulo $\kappa_n\frac{(2\pi i)^n}{(n-1)!}$ on $\widehat\C_{-+}$, and modulo $\frac{(\pi i)^n}{(n-1)!}$ on $\widehat\C_{--}$.
\end{corollary}
\begin{proof}
For $\widehat\C_{++}$, $\widehat\C_{+-}$ and $\widehat\C_{--}$ this follows immediately from Lemma~\ref{lemma:DeltaFormulas}. When expanding $\delta(p+\frac{1}{2},n)$ one easily verifies that the greatest common divisor of the denominators is $2^{n-2}$ if $n$ is even and $2^{n-3-\nu_2(n-1)}$ if $n$ is odd. This proves the case $\widehat\C_{-+}$.
\end{proof}
\begin{lemma} $\widehat L_n$ is a primitive for the one form $\omega_n$ defined in \eqref{eq:wnDef}, i.e.~$d\widehat L_n=\omega_n$.
\end{lemma}
\begin{proof}
By~\eqref{eq:LiDef} one has $d\Li_k(z)=\frac{\Li_{k-1}(z)}{z}dz$, and it follows that $d\L i_k(z;q)=\frac{\L i_{k-1}(z;q)}{z}dz$. This holds for all $k\geq 0$ with the convention that $\L i_0(z;q)=\frac{z}{1-z}$. Letting $u=\Log(\epsilon_1z;p)$, $v=\Log(\epsilon_2(1-z);q)$, $\widehat L_n=\widehat L_n(\langle z;p,q\rangle_{\epsilon_1\epsilon_2})$, and $\L i_k=\L i_k(z;q)$ one has
\begin{equation}
\begin{aligned}
d\widehat L_n&=\sum_{r=0}^{n-1}\frac{(-1)^r}{r!}(\L i_{n-r-1}u^r+r\L i_{n-r}u^{r-1})\frac{dz}{z}-\frac{(-1)^n}{n!}((n-1)u^{n-2}vdu+u^{n-1}dv)\\&=\frac{(-1)^{n-1}}{(n-1)!}u^{n-1}\L i_0\frac{dz}{z}-\frac{(-1)^n}{n!}((n-1)u^{n-2}vdu+u^{n-1}dv)=\omega_n.
\end{aligned}
\end{equation}
The second equality follows by telescoping,
%using that $\sum_{r=0}^{n-1} \frac{(-1)^r}{r!}(a_{r+1}+ra_r)=\frac{(-1)^{n-1}}{(n-1)!}a_n$, 
and the third from the fact that $\L i_0(z;q)\frac{dz}{z}\!=\!-dv$.
\end{proof}

\begin{theorem}\label{thm:LnhatProofSection} The form $\omega_n$ has periods in $\frac{(2\pi i)^n}{(n-1)!}\Z$ on $\widehat\C_{++}$ and $\widehat\C_{+-}$ and in $\kappa_n\frac{2\pi i}{(n-1)!}$ on $\widehat\C_{-+}$ and $\widehat\C_{-+}$, where $\kappa_n$ is defined in~\eqref{eq:KappaDef}.
\end{theorem}
\begin{proof}
The commutator subgroup of $\pi_1(\C\setminus\{0,1\})$ is generated by the loops $\gamma_{k,l}=a^kb^la^{-k}b^{-l}$ where $a$ is a loop going counterclockwise around 0 and $b$ is a loop going clockwise around 1.
It is thus enough to compute the integral of $\omega_n$ along a lift of $\gamma_{k,l}$. Since $\widehat L_n$ is a primitive, this is always of the form $\frac{2\pi i}{(n-1)!}A_{k,l}$ where $A_{k,l}$ is an integral linear combination of terms $\Delta^I_{\epsilon_1\epsilon_2}(r,s)$. One easily checks that when $\epsilon_1=1$, $A_{k,l}\in\Z$, and when $\epsilon_1=-1$, $A_{k,l}\in\kappa_n\Z$.  For example, if $k=l=1$ we have
\begin{multline}
A_{1,1}=
-\Delta_{\epsilon_1\epsilon_2}^{(-\infty,0)}(p,q)+\Delta_{\epsilon_1\epsilon_2}^{(0,1)}(p+1,q)-\Delta_{\epsilon_1\epsilon_2}^{(1,\infty)}(p+1,q)+\\\Delta_{\epsilon_1\epsilon_2}^{(-\infty,0)}(p,q-1)-\Delta_{\epsilon_1\epsilon_2}^{(0,1)}(p,q-1)+\Delta_{\epsilon_1\epsilon_2}^{(1,\infty)}(p,q),
\end{multline}
which equals $-\delta(p+1,n)\in\Z$ when $\epsilon_1=1$ and $-\delta(p+\frac{1}{2},n)\in \kappa_n\Z$ when $\epsilon_1=-1$.
\end{proof}

\subsubsection{Modifying $\widehat L_n$ on $\widehat\C_{--}$} We now define $\widehat\L_n$ to be $\widehat L_n$ except on $\widehat\C_{--}$, where we define it as (with the convention that $\Imag(z+0i)>0$ and $\Imag(z-0i)<0$ for $z\in\R$) 
\begin{equation}
\widehat\L_n =\widehat L_n+\frac{(2\pi i)^n}{(n-1)!}\begin{cases}
-(-p-\frac{1}{2})^{n-1}+(-\frac{1}{2})^{n-1}&\text{if }\enspace\Imag(z)>0\\(-\frac{1}{2})^{n-1}&\text{if }\enspace\Imag(z)<0.\end{cases}
\end{equation}
This is well defined since the imaginary part is never zero on $\overline \C^{\cut}_{--}$.
The fact that $\widehat\L_n$ is defined modulo $\kappa_n\frac{2\pi i}{(n-1)!}$ on $\widehat\C_{--}$ is an easy consequence of Lemma~\ref{lemma:DeltaFormulas}.

\section{$\widehat\L_n$ relations and the symbol map}\label{sec:RelsAndSymb}
To motivate our treatment of $\widehat\L_n$ relations we begin with a review of the so-called \emph{symbol} map. %If the symbol vanishes, you get a relation for $\L_n$ (see Theorem~\ref{thm:ZeroSymb} below for a precise statement). 
The new idea is that in a certain sense, the vanishing of the symbol map is equivalent to the vanishing of a certain 1-form, which is a multiple of the 1-form $d\widehat\L_n$ (see Proposition~\ref{prop:SymbEqDiff}).

\subsection{The symbol map} The \emph{symbol map} is the map ($[0]$, $[1]$, and $[\infty]$ map to 0)
\begin{equation}
\symb_n\colon \Z[P^1_F]\to\wedge^2(F^*)\otimes \Sym^{n-2}(F^*),\qquad [z]\mapsto z\wedge(1-z)\otimes z^{\otimes(n-2)}.
\end{equation}
It follows from Goncharov's definition of $\Gamma(F,n)$ that $\symb_n$ factors through $\B_n(F)$, and that if $\beta\in\Z[P^1_F]$ satisfies $\delta(\beta)=0\in\B_{n-1}\otimes F^*$, then $\symb_n(\beta)=0$. The converse is false.

The two results below link the symbol map to functional relations for $\L_n$.
\begin{theorem}[{Zagier~\cite[Prop.~3]{ZagierPolylogs}, Goncharov~\cite[Thm.~1.17]{GoncharovMotivicGalois}}]\label{thm:ZeroSymb} Let $\beta\in\Z[P^1_{\C(t)}]$. If $\symb_n(\beta)=0$, then the $\L_n(\beta)$ (regarded as a function in $t$) is constant. Moreover, if $\L_n(\beta)$ is identically 0, then $\delta(\beta)=0$, so $\beta(t)\in\Z[P^1_\C]$ is constant in $\B_n(\C)$.
\end{theorem}

%The condition $\symb_n(\beta)=0$ does not necessarily imply that $\delta(\beta)=0\in \B_{n-1}(F)\otimes F^*$. We give a sufficient condition for this to happen below. It primarily serves to motivate later discussions. 

The following discussion serves mainly to motivate later definitions. Let $\beta=\sum_{i=1}^Mr_i[z_i]\in\Z[F(t)]$ and suppose that for $a_1,\dots,a_N\subset F(t)^*$ and integers $k_{ji}$ and $l_{ji}$ we have 
\begin{equation}\label{eq:zAnd1Minusz}
z_i=a_1^{k_{1i}}\cdots a_N^{k_{Ni}},\qquad 1-z_i=a_1^{l_{1i}}\cdots a_N^{l_{N i}}.
\end{equation}
For each integer $0<l<n-1$ and each multisubset (elements need not be distinct) $J=\{j_1,\dots,j_{l}\}$ of $\{1,\dots,N\}$ define
\begin{equation}
\pi_J=\sum_{i=1}^M r_ik_{j_1i}\cdots k_{j_{l}i}[z_i]\in\Z[F(t)].
\end{equation} 
The following is an easy induction argument using that $z_i^{\otimes l}=\sum_{|J|=l} k_{j_1i}\cdots k_{j_{l}i}a_{j_1}\otimes\dots\otimes a_{j_{l}}$.

\begin{proposition}\label{prop:LowerLevelMotivation}
Suppose $\symb_n(\beta)=0$. If for some $t_0$, the elements $\pi_J(t_0)$ are zero in $\B_{n-l}(F)$ for all $J$ with $\vert J\vert=l$, then then $\delta(\beta)=0\in\B_{n-1}(F(t))\otimes F(t)^*$. In particular, it follows that
\begin{equation}
\beta(t)-\beta(t_0)\in R_n(F) \text{ for all } t.
\end{equation}
\end{proposition}

\subsection{Our setup}
Let $a_i$ and $\widetilde a_i$ be formal variables. We think of $\widetilde a_i$ as a logarithm of $a_i$. 
Consider the polynomial rings
\begin{equation}\label{eq:SandStilde}
S=\Z[a_1^{\pm 1},a_2^{\pm 1},\dots], \qquad \widetilde S=\Z[\widetilde a_1,\widetilde a_2,\dots].
\end{equation}
Let $\widetilde S_k\subset\widetilde S$ denote the group of homogeneous polynomials of degree $k$, and let $U$ denote the free multiplicative group on the $a_i$. We have a canonical group homomorphism
\begin{equation}
\pi\colon \widetilde S_1\to U,\qquad \widetilde a_i\mapsto a_i.
\end{equation}
We shall consider elements $\alpha\in\Z[\widetilde S_1\times \widetilde S_1]$. Each such can be canonically written as $\sum_{i=1}^M r_i(u_i,v_i)$ where $r_i\in\Z$ and 
\begin{equation}\label{eq:uivi}
u_i=\sum_{j=1}^Nk_{ji}\widetilde a_j\in\widetilde S_1,\quad v_i=\sum_{j=1}^Nl_{ji}\widetilde a_j\in\widetilde S_1
\end{equation}

\begin{convention} For $\alpha\in\Z[\widetilde S_1\times \widetilde S_1]$ we always define $u_i$ and $v_i$ as in~\eqref{eq:uivi}. We similarly define $r=(r_1,\dots,r_M)$, $K=\{k_{ji}\}$, $L=\{l_{ji}\}$, and
\begin{equation}
 z_i=\pi(u_i)=\prod_{j=1}^Na_j^{k_{ji}}\in U,\quad w_i=\pi(v_i)=\prod_{j=1}^Na_j^{l_{ji}}\in U.
 \end{equation}
\end{convention}

\subsection{Realizations} Let $\alpha\in\Z[\widetilde S_1\times \widetilde S_1]$.

\begin{definition}\label{def:RealizationScheme}  The \emph{realization scheme} for $\alpha$ is the scheme $X_\alpha$ defined by the \emph{realization equations}
\begin{equation}\label{eq:realizationIdealEqs}
\prod_{j=1}^Na_j^{k_{ji}}+\prod_{j=1}^Na_j^{l_{ji}}-1\in \Z[a_1^{\pm 1},\dots,a_N^{\pm 1}], \qquad i=1,\dots, M.
\end{equation}
More precisely $X_\alpha=\Spec(\Z[a_1^{\pm 1},\dots,a_N^{\pm 1}]/I_\alpha)$ where $I_{\alpha}$ is the ideal generated by the realization equations. For a field $F$ we let $X_\alpha(F)=\Spec(F[a_1^{\pm 1},\dots,a_N^{\pm 1}]/I_\alpha)$. 
\end{definition}
Note that $X_\alpha$ only depends on $K$ and $L$.
\begin{convention} By a point $p$ in $X_\alpha(F)$ we always mean a rational point, i.e.~a ring homomorphism $p\colon\Z[a_1^{\pm 1},\dots,a_N^{\pm 1}]\to F$ killing the realization ideal. We also assume that $p$ is smooth. We extend $p$ to $S$ by mapping $a_i$ to 1 for $i>N$. Then $p$ restricts to a group homomorphism $p_U\colon U\to \C^*$.
\end{convention}

\begin{definition}\label{def:logpoint} Let $p$ be a point in $X_\alpha(\C)$. A \emph{lift} of $p$ is a homomorphism $\widetilde p\colon\widetilde S_1\to\C$ lifting $p_U$ in the sense that $\exp\circ\widetilde p=p_U\circ \pi$. A point together with a lift is called a \emph{log-point}.
\end{definition}

A log-point $\widetilde p$ in $X_\alpha(\C)$ determines elements
\begin{equation}\label{eq:pandptilde}
\widetilde p(\alpha)=\sum_{i=1}^M r_i[(\widetilde p(u_i),\widetilde p(v_i))]\in\Z[\widehat\C],\qquad p(\alpha)=\sum_{i=1}^M r_i[z_i]\in\Z[\C\setminus\{0,1\}]\subset\Z[P^1_\C]
\end{equation}
such that $r(\widetilde p(\alpha))=p(\alpha)$, where $r$ is the covering $\widehat\C\to\C\setminus\{0,1\}$.
We shall think of $\alpha$ as a purely symbolic representation of a Bloch group element, and we refer to $\widetilde p(\alpha)$ and $p(\alpha)$ as \emph{realizations} of $\alpha$ in $\Z[\widehat\C]$ and $\Z[P^1_\C]$, respectively.
If $Y\subset X_\alpha(\C)$ is a smooth submanifold and $\widetilde p_Y$ is a family of log-points $Y$, \eqref{eq:pandptilde} gives rise to maps
\begin{equation}\label{eq:ptildeYandpY}
\widetilde p_Y(\alpha)\colon Y\to\Z[\widehat\C],\qquad p_Y(\alpha)\colon Y\to \Z[\C\setminus\{0,1\}].
\end{equation}
In particular, $\widehat\L_n\circ \widetilde p_Y(\alpha)$ and $\L_n\circ p_Y(\alpha)$ are functions on $Y$. If the logarithms of the coordinates are smooth, we say that $\widetilde p_Y$ is a \emph{smooth family of log-points over $Y$}. If so, $\widehat\L_n\circ \widetilde p_Y(\alpha)$ is smooth.

\subsection{Differential $\widehat\L_n$ relations and the symbol}
For an integer $k>0$ let $\Omega^1_k(\widetilde S)$ denote the group of 1-forms on $S$ of degree $k$ (finite formal sums of terms $f_I d\widetilde a_i$ where $f_I$ is a degree $k$ monomial). 
\begin{definition}\label{def:wnalpha} We say that $\alpha$ is a \emph{differential $\widehat\L_n$ relation} if
\begin{equation}
w_n(\alpha)=\sum_{i=1}^{M} r_iu_i^{n-2}(u_idv_i-v_idu_i)
\end{equation}
is $0$ in $\Omega^1_{n-1}(\widetilde S)$.
\end{definition}

\begin{proposition}\label{prop:LnhatRels} Suppose $\alpha$ is a differential $\widehat\L_n$ relation. For any smooth family $\widetilde p_Y$ of log-points over a connected $Y\subset X_\alpha(\C)$ the function $\widehat\L_n\circ\widetilde p_Y(\alpha)$ is constant on $Y$.
\end{proposition}
\begin{proof}
This follows from the fact that $d\widehat\L_n(u,v)=(-1)^n\frac{n-1}{n!}u^{n-2}(udv-vdu)$.
\end{proof}

\begin{remark}
Although any differential $\widehat\L_n$ relation with $\dim (X_\alpha(\C))>0$ provides local $\widehat\L_n$ relations, the value of the constant $\widehat\L_n\circ\widetilde p_Y(\alpha)$ may depend dramatically on the choice of smooth logarithms (see Example~\ref{ex:LogDependence} for an example).%This should be clear from Proposition~\ref{prop:LhatDifference};. %We shall mainly be interested in examples where the value is independent of the choice of logarithms. It turns out that such have purely combinatorial descriptions allowing us to define relations for arbitrary fields.
\end{remark}

The result below relates our notion of differential $\widehat\L_n$ relation to the vanishing of the symbol.
\begin{proposition}\label{prop:SymbEqDiff}
The 1-form $w_n(\alpha)$ vanishes in $\Omega^1_{n-1}(\widetilde S)$ if and only if the element $A=\sum_{i=1}^Mr_i(z_i\wedge w_i)\otimes z_i^{n-2}$ in $\wedge^2(U)\otimes\Sym^{n-2}(U)$ has order $1$ or $2$.
\end{proposition}
\begin{proof}
Consider the composition
\begin{equation}
\wedge^2(U)\otimes\Sym^{n-2}(U)\overset{\cong}{\to}\wedge^2(\widetilde S_1)\otimes\Sym^{n-2}(\widetilde S_1)\overset{\cong}{\to}\wedge^2(\widetilde S_1)\otimes \widetilde S_{n-2}\hookrightarrow \Omega^1_{n-1}(\widetilde S).
\end{equation}
The left isomorphism is induced by the canonical map $U\to \widetilde S_1$ taking $a_i$ to $\widetilde a_i$, the middle isomorphism is induced by the canonical identification $\Sym^{k}(\widetilde S_1)\cong \widetilde S_k$, and the right map takes $\widetilde a_i\wedge\widetilde a_j\otimes f$ to $f(\widetilde a_id\widetilde a_j-\widetilde a_jd\widetilde a_i)$. The right map is injective modulo 2-torsion, and the composition takes $A$ to $w_n(\alpha)$. This proves the result.
\end{proof}

\begin{corollary}
Suppose $\beta=\sum_{i=1}^Mr_i[z_i]$ is in the kernel of $\symb_n$, and suppose that there are multiplicatively independent functions $a_j\in\C(x)^*$ such that~\eqref{eq:zAnd1Minusz} holds. Any choice of local log branches of the $a_i$ determines a local $\widehat\L_n$ relation 
\begin{equation}
\sum_{i=1}^M\widehat\L_n(u_i,v_i)=c,
\end{equation}% \hfill{\qedsymbol}
where $c$ is a constant, $u_i$ and $v_i$ are given by \eqref{eq:uivi} with $\widetilde a_i=\log(a_i)$.
\end{corollary}

\begin{remark}
There are many examples where~\eqref{eq:zAnd1Minusz} only holds up to signs. For example the element $[x]+(-1)^n[x^{-1}]\in\B_n(\C(x))$ or $[x]+[\frac{1}{1-x}]+[1-x^{-1}]\in\B_3(\C(x))$. The same is true for Goncharov's element $R(x,y,z)$ (see Conjecture~\ref{conj:GonConj0}~\ref{R3Conj}), Gangl's 931 term relation~\cite{Gangl931} and many others. As we see in 
section~\ref{sec:CpmRealizations} such $\L_n$ relations have lifts to $\widehat\C_{\pm}$ (but not $\widehat\C$).
\end{remark}

\subsection{Examples}\label{sec:Examples} For notational convenience, we shall occasionally use \emph{multiplicative notation} to denote elements in $\widetilde S_1\times\widetilde S_1$, e.g.~we write $ (a_1a_2,a_3a_4)$ instead of $(\widetilde a_1+\widetilde a_2,\widetilde a_3+\widetilde a_4)$. We shall also occasionally denote the free variables by other symbols than $a_i$.

\begin{example}\label{ex:BasicL2Rel}
The element $[(\widetilde a_1,\widetilde a_2)]+[(\widetilde a_2,\widetilde a_1)]$ is a differential $\widehat\L_2$ relation. Clearly, $X_{\alpha}(\C)=\C\setminus\{0,1\}$ and a lift of a point $x\in X_{\alpha}(\C)$ is a pair of complex numbers $(u,v)$, with $e^u=x$ and $e^v=1-x$. The corresponding relation is
\begin{equation}\label{eq:BasicPlusRel}
\widehat\L_2(u,v)+\widehat\L_2(v,u)=\frac{\pi^2}{6}\in\C/4\pi^2\Z,
\end{equation}
which we regard as a lift of the relation $\L_2(x)+\L_2(1-x)=0$.
\end{example}

\begin{example}[The lifted five term relation]\label{ex:LiftedFTSymb}  
The element
\begin{equation}
\alpha=[(a_1,a_3)]-[(a_2,a_4)]+[(\frac{a_2}{a_1},\frac{a_5}{a_1})]-[(\frac{a_2a_3}{a_1a_4},\frac{a_5}{a_1a_4})]+[(\frac{a_3}{ a_4},\frac{a_5}{a_4})]
\end{equation}
is a differential $\widehat\L_2$ relation. One easily checks that $X_\alpha(\C)=\{(x,y)\in\C\setminus\{0,1\}\bigm\vert x\neq y\}$ with 
\begin{equation}
a_1=x,\quad a_2=y,\quad a_3=1-x,\quad a_4=1-y,\quad a_5=x-y.
\end{equation}
For each log-point $\widetilde p$, the realization $\widetilde p(\alpha)$ is an instance of Neumann's lifted five term relation~\cite[Def.~3.2]{ZickertAlgK}. It thus follows (recall that Neumann's $R$ equals $\widehat\L_2-\frac{\pi^2}{6}$) that $\widehat\L_2(\widetilde p(\alpha))=\frac{\pi^2}{6}$. 
\end{example}

\begin{example}[The (inverted) lifted five term relations]\label{ex:InvFTSymb}
By Example~\ref{ex:BasicL2Rel}
\begin{equation}
\beta=-[(a_3,a_1)]+[(a_4,a_2)]-[(\frac{a_5}{a_1},\frac{a_2}{a_1})]+[(\frac{a_5}{a_1a_4},\frac{a_2a_3}{a_1a_4})]-[(\frac{a_5}{a_4},\frac{a_3}{ a_4})]
\end{equation}
is also a differential $\widehat\L_2$ relation with $\widehat\L_2(\widetilde p(\beta))=0\in\C/4\pi^2\Z$ for each log-point $\widetilde p$.
\end{example}

\begin{example}[A 31 term relation]\label{ex:31TermRel} The free variables $a_i$ with $i\in\{1,2,3\}$, and $b_J$ with $\emptyset\neq J\subset\{1,2,3\}$ admit a natural cyclic $\Z_3$ action. The 31 term element
\begin{equation}
\alpha=[(a_1a_2a_3,b_{123})]+\sum_{\sigma\in\Z_3}\sigma A,
\end{equation}
where $A$ is given by
\begin{equation}
\begin{gathered}
[(a_1,b_1)]-[(\frac{b_1}{b_{12}},\frac{a_1b_2}{b_{12}})]-[(\frac{b_1}{b_{13}},\frac{a_1b_3}{b_{13}})]+[(\frac{b_1}{b_{123}},\frac{a_1b_{23}}{b_{123}})]-[(a_1a_2,b_{12})]+\\
-[(\frac{a_1b_2}{b_{12}},\frac{b_1}{b_{12}})]-[(\frac{a_1b_3}{b_{13}},\frac{b_1}{b_{13}})]+[(\frac{b_1b_{123}}{b_{12}b_{13}},\frac{a_1b_2b_3}{b_{12}b_{13}})]+
[(\frac{a_1a_2b_3}{b_{123}},\frac{b_{12}}{b_{123}})]+[(\frac{a_1b_2b_3}{b_{12}b_{13}},\frac{b_1b_{123}}{b_{12}b_{13}})]
\end{gathered}
\end{equation}
is a differential $\widehat\L_3$ relation.
The points in $X_\alpha(F)$ may be identified with triples $(x_1,x_2,x_3)\in F^*$ all products of which are distinct from 1, with the identification $a_i=x_i$, $b_J=\prod_{j\in J}x_j$. We shall see later (Example~\ref{ex:R31InR3}) that when $F=\C$, the relation $\widehat\L_3(\widetilde p(\alpha))$ is zero in $\C/\frac{(2\pi i)^3}{2}\Z$ for each log-point $\widetilde p$.
\end{example}

\begin{remark}
After a change of variables $(x_1,x_2,x_3)=(\frac{y_2(1-y_1(1-y_3))}{1-y_3(1-y_2)},y_3,\frac{1}{1-y_1(1-y_3)})$ the element $p(\alpha)$ is an instance of Goncharov's 22 term element $R(y_1,y_2,y_3)$ up to instances of $[x]=[x^{-1}]$ and $[x]+[1-x]+[1-1/x]=[1]$.
\end{remark}

\subsection{Realizations in $\widehat\C_{\pm}$}\label{sec:CpmRealizations}
To define realizations of $\alpha$ in $\widehat\C_{\pm}$ we need additional data in the form of a \emph{sign determination}, which is a vector $\mathcal V=\big((\epsilon_{1,1},\epsilon_{2,1}),\dots,(\epsilon_{1,M},\epsilon_{2,M})\big)$ of sign pairs. Given such we have a realization scheme $X_{\alpha,\mathcal V}$ defined as in Definition~\ref{def:RealizationScheme}, using the ideal generated by $\epsilon_{1,i}\pi(u_i)+\epsilon_{2,i}\pi(v_i)-1$. If $\widetilde p$ is a lift of a point in $X_{\alpha,\mathcal V}$, we can define $\widetilde p(\alpha)\in\Z[\widehat\C_{\pm}]$ and $p(\alpha)\in\Z[\C\setminus\{0,1\}$ as in~\eqref{eq:pandptilde} (but with $z_i=\epsilon_{1,i}\pi(u_i)$). Note that $(\widetilde p(u_i),\widetilde p(v_i))$ is in the component $\widehat\C_{\epsilon_{1,i}\epsilon_{2,i}}$ of $\widehat\C_{\pm}$. We shall thus sometimes denote $(u_i,v_i)\in\widetilde S_1\times\widetilde S_1$ by $(u_i,v_i)_{\epsilon_{1,i},\epsilon_{2,i}}$.

\begin{example}\label{ex:2term}
$[(\widetilde a_1,\widetilde a_2)_{-+}]+(-1)^n[(-\widetilde a_1,\widetilde a_2-\widetilde a_1)_{-+}]$ is a differential $\widehat \L_n$ relation. The corresponding relation is an instance of~\eqref{eq:Tauhat}.
\end{example}

\begin{example}\label{ex:3term}
$[(\widetilde a_1,\widetilde a_2)]+[(-\widetilde a_2,\widetilde a_1-\widetilde a_2)]+[(\widetilde a_2-\widetilde a_1,-\widetilde a_1)]$ is a differential $\widehat\L_3$ relation. The corresponding $\widehat \L_3$ relation for $\mathcal V=\big((1,1),(1,-1),(-1,1)\big)$ or any cyclic permutation of $\mathcal V$ is an instance of~\eqref{eq:Sigmahat}.
\end{example}

\begin{example}[Goncharov's 22 term relation]\label{ex:LiftedGonRel} Consider free variables $\alpha_i$, $\beta_i$, $\gamma_i$ with $i\in\{1,2,3\}$ and an additional variable $\delta$. Let $\Z_3$ act by cyclic permutation on $\alpha_i$, $\beta_i$  and $\gamma_i$ and trivially on $\delta$. The element
\begin{equation}
[(\alpha_1\alpha_2\alpha_3,\delta)_{-+}]+\sum_{\sigma\in \Z_3} \sigma(A),
\end{equation}
where $A$ is given by
\begin{equation}
\begin{gathered}
[(\alpha_i,\gamma_i)_{++}]+[(\beta_i,\alpha_i\gamma_{i-1})_{++}]-[(\frac{\alpha_{i-1}}{\beta_i},\frac{\gamma_{i-1}\gamma_i}{\beta_i})_{++}]+[(\frac{\beta_i}{\alpha_{i-1}\alpha_i},\frac{\gamma_i}{\alpha_{i-1}\alpha_i})_{+-}]+\\
[(\frac{\alpha_i\beta_{i-1}}{\beta_{i+1}},\frac{\gamma_{i+1}\beta_i}{\beta_{i+1}})_{++}]+[(\frac{\beta_i}{\alpha_i\beta_{i-1}},\frac{\delta}{\alpha_i\beta_{i-1}})_{-+}]-[(\frac{\alpha_{i-1}\alpha_i\beta_{i+1}}{\beta_i},\frac{\delta\gamma_i}{\beta_i})_{++}],
\end{gathered}
\end{equation}
is a differential $\widehat\L_3$ relation. The $F$ points in $X_{\alpha,\mathcal V}$ may be identified with 
\begin{equation}
\left\{(y_1,y_2,y_3)\in(F\setminus\{0,1\})^3\bigm\vert y_i(1-y_{i-1})\neq 1 \text{ (indices mod 3)},\quad y_1y_2y_3\neq -1\right\}
\end{equation}
with $\alpha_i=y_i$, $\beta_i=1-y_i(1-y_{i-1})$, $\gamma_i=1-y_i$, and $\delta=1+y_1y_2y_3$. For each point in $p\in X_{\alpha,\mathcal V}(\C)$, $p(\alpha)$ is an instance of Goncharov's 22 term relation, so we have $\L_3(p(\alpha))=3\zeta(3)$ (see~\cite[p.~428]{ZagierPolylogs}). One can show that $\widehat\L_3(\widetilde p(\alpha))=3\zeta(3)$ holds as well (see Remark~\ref{rm:Gon22Lift}).
\end{example}

In the previous examples the choice of lift made no difference to the value of $\widehat\L_n(\widetilde p(\alpha))$. In other words, logarithms of $a_i$ could be chosen independently and arbitrarily without affecting the $\widehat\L_n$ relation. The next example shows that this is not always the case. 
\begin{example}\label{ex:LogDependence}
$\alpha=[(2\widetilde a_1+2\widetilde a_3,-\widetilde a_1+\widetilde a_2)]-2^{n-1}[(\widetilde a_1+\widetilde a_3,-2\widetilde a_1+\widetilde a_2-\widetilde a_3)]$ is a differential $\widehat\L_n$ relation. One checks that $X_{\alpha,\mathcal V}(\C)$ is empty for $\mathcal V=\big((1,1),(1,1)\big)$, so there are no realizations in $\widehat\C$. For $\mathcal V=\big((-1,1),(1,1)\big)$, $X_{\alpha,\mathcal V}(\C)$ is given by
\begin{equation}\label{eq:Additiveatilderels}
a_1=x, \qquad a_2=\frac{1}{2}(\omega+1)x,\qquad a_3=\frac{\omega}{x},\qquad x\in\C\setminus\{0\},\qquad 2\omega^2-\omega+1=0.
\end{equation}
For any log-point $\widetilde p$, $p(\alpha)=[-\omega^2]-2^{n-1}[\omega]\in\Z[\C\setminus\{0,1\}]$, which only depends on $\omega$. However, one easily checks (e.g.~numerically) that different lifts give different values of $\widehat\L_n(\widetilde p(\alpha))$ when $n>2$. We shall not need the values, so we omit them.
\end{example}

\subsection{Realizations for arbitrary fields} Fix a $\Z$-extension 
\begin{equation}
0\to \Z\overset{\iota}{\to} E\overset{\pi}{\to}F^*\to 0
\end{equation}
of $F^*$ such that $E$ is torsion free. If $F=\C$ we choose the extension given by the exponential function. We can then define (omitting $E$ from the notation)
\begin{equation}
\widehat F=\Big\{(u,v)\in E\times E\bigm\vert \pi(u)+\pi(v)=1\Big\}.
\end{equation}
One can define log-points in $X_{\alpha}(F)$ as lifts of points with values in $E$ exactly as in Definition~\ref{def:logpoint}. For any $\alpha\in\Z[\widetilde S_1\times\widetilde S_1]$ a log-point $\widetilde p$ determines realizations $\widetilde p(\alpha)\in\Z[\widehat F]$ and $p(\alpha)\in\Z[P^1_F]$ with $r(\widetilde p(\alpha))=p(\alpha)$ as in~\eqref{eq:pandptilde}. Here $r\colon\widehat F\to F\setminus\{0,1\}$ takes $(u,v)$ to $\pi(u)$.

\subsubsection{The other variant} If $F$ does not have characteristic 2 (equivalently if $-1\neq 1$ in $F$), we may also define $\widehat F_{\pm}$ as above, but with $\epsilon_1\pi(u)+\epsilon_2\pi(v)=1$ and  $r\colon\widehat F_{\pm}\to F\setminus\{0,1\}$ taking $(u,v)$ to $\epsilon_1\pi(u)$. A log-point  in $X_{\alpha,\mathcal V}(F)$ then gives rise to a realization of $\alpha$ in $\Z[\widehat F_{\pm}]$.

\section{The lifted Bloch complexes}\label{sec:LiftedBloch}
We now define a complex $\widehat\Gamma(F,n)$ lifting Goncharov's complex $\Gamma(F,n)$. 
\subsection{Review of Goncharov's construction of $\Gamma(F,n)$} Goncharov's complex is defined in terms of groups $\B_k(F)=\Z[P^1_F]/R_k(F)$, where the $R_k(F)\subset\Z[P^1_F]$ are defined inductively starting with $k=2$. For a field $K$, let $A_2(K)$ denote the kernel of the map $\delta\colon\Z[P^1_K]\to\wedge^2(K^*)$ taking $[z]$ to $z\wedge(1-z)$ (and $[0]$, $[1]$ and $[\infty]$ to 0). Then $R_2(F)$ is generated by $[0]$, $[\infty]$, and elements of the form $p(\alpha)-q(\alpha)$, with $\alpha\in\A_2(F(Y))$, and $p$ and $q$ are points on a geometrically irreducible smooth curve $Y$ over $F$ with function field $F(Y)$. If $R_{k-1}(K)$ has been defined for all fields $K$, there is a map $\delta\colon\Z[P^1_K]\to \B_{k-1}(K)\otimes K^*$ taking $[z]$ to $[z]\otimes z$ (and $[0]$ and $[\infty]$ to 0). Letting $\A_k(K)$ denote its kernel, $R_k(F)$ is defined as above, but with $\A_k(F(Y))$ instead of $\A_2(F(Y))$. One then shows that the $\delta$ maps induce maps 
\begin{equation}
\delta\colon\B_k(F)\to\B_{k-1}(F)\otimes F^*, \quad k>2,\qquad \delta\colon\B_2(F)\to\wedge^2(F^*),
\end{equation}
which induce boundary maps $\delta\colon\B_k(F)\otimes \wedge^l(F^*)\to \B_{k-1}(F)\otimes\wedge^{l+1}(F^*)$, taking $[z]\otimes a$ to $[z]\otimes z\wedge a$ for $k>2$, and $\B_2(F)\otimes \wedge^l(F^*)\to\wedge^{l+2}(F^*)$ taking $[z]\otimes a$ to $z\wedge (1-z)\wedge a$. This completes the construction. 

\begin{remark} Goncharov implicitly assumes that $x\wedge (-x)=x\wedge x=0\in\wedge^2(F^*)$ and that $[1]=0\in\B_2(F)$. It then follows that $[x]+[x^{-1}]$ and $[x]+[1-x]$ are in $R_2(F)$ for any $x\in F$.
\end{remark}

\subsection{Overview of our construction and main results}
In lifting Goncharov's construction we face two obstacles: Firstly, Goncharov allows $0$, $1$ and $\infty$ which have no lifts, and secondly, $\widehat\L_n(\widetilde p(\alpha))$ depends on the log-branches in a seemingly non-algebraic way. The first obstacle is addressed by considering a notion of \emph{permissible lifts} of \emph{zero-degenerate points} (Definition~\ref{def:ZeroDegLift}). The second obstacle is addressed by introducing the notion of log-points \emph{killing lower levels} (Definition~\ref{def:KillLower}), an inductive definition inspired by Proposition~\ref{prop:LowerLevelMotivation}. It turns out that if $\widetilde p$ kills the lower levels of $\alpha$, then $\widehat\L_n(\widetilde p(\alpha))$ is independent of the logarithms modulo $\frac{(2\pi i)^n}{n!}$, and in order to be well defined modulo $\frac{(2\pi i)^n}{(n-1)!}$ we need the additional concept of \emph{proper ambiguity} (Definition~\ref{def:ProperAmbiguity}), a non-inductive purely symbolic property. All the above concepts may be defined over an arbitrary field.

We define $\widehat\B_n(\widehat F)=\Z[\widehat F]/\widetilde R_n(F)$, where $\widetilde R_n(F)$ is defined as follows.

\begin{definition}\label{def:RntildeDef} The set $\widetilde R_n(F)$ is the subset of $\Z[\widehat F]$ generated by the following two types of relations, where $\alpha$ denotes a differential $\widehat\L_n$ relation with proper ambiguity.
\begin{enumerate}
\item $\widetilde p(\alpha)-\widetilde p_0(\alpha)$, where $\widetilde p_0(\alpha)$ kills the lower levels of $\alpha$, and $\widetilde p$ is a lift of a point in the same geometric component of $p_0$ in $X_\alpha(F)$.
\item $\widetilde p(\alpha)$, where the geometric component of $p$ in $\overline{X_\alpha(F)}$ contains a zero-degenerate point with a permissible lift.\label{item:2}
\end{enumerate}
\end{definition}

\begin{remark}
Goncharov does not need \eqref{item:2} since if $p_0(\alpha)\in\Z[\{0\}]$ for $\alpha\in\A_n(F(Y))$, then $p(\alpha)\in R_n(F)$. This is because $[0]\in R_n(F)$ by definition.
\end{remark}

There are homomorphisms
\begin{equation}\label{eq:NunHatDef}
\begin{aligned}
&\delta\colon\Z[\widehat F]\to\widehat\B_{n-1}(\widehat F)\otimes E,&&[(u,v)]\mapsto [(u,v)]\otimes u, &&n>2\\
& \delta\colon\Z[\widehat F]\to\wedge^2(E),&&[(u,v)]\mapsto u\wedge v.
\end{aligned}
\end{equation}

The proofs of the two next theorems are purely formal.

\begin{theorem}[Proof in Section~\ref{sec:RnTildeDef}]\label{thm:RntildeToRn}
The projection $\Z[\widehat F]\to\Z[F\setminus\{0,1\}]$ induced by $\pi\colon E\to F^*$ takes $\widetilde R_n(F)$ to $R_n(F)$.
\end{theorem}

\begin{theorem}[Proof in Section~\ref{sec:RnTildeDef}]\label{thm:Nuhat} The $\delta$ map takes $\widetilde R_n(F)$ to $0$ and thus descend to homomorphisms 
\begin{equation}
\delta\colon\widehat\B_n(\widehat F)\to\widehat\B_{n-1}(\widehat F)\otimes E,\qquad \delta\colon\widehat\B_2(\widehat F)\to\wedge^2(E).
\end{equation}
for $n>2$ and $n=2$, respectively.
\end{theorem}

By Theorem~\ref{thm:Nuhat} we have a chain complex $\widehat\Gamma(F,n)$:
\begin{equation}
\cxymatrix{{{\widehat\B_n(\widehat F)}\ar[r]^-{\delta_1}&\cdots\ar[r]^-{\delta_k}&{\widehat\B_{n-k}(\widehat F)\otimes\wedge^k(E)}\ar[r]^-{\delta_{k+1}}&\cdots\ar[r]^-{\delta_{n-2}}&{\widehat\B_2(\widehat F)\otimes\wedge^{n-2}(E)}\ar[r]^-{\delta_{n-1}}&{\wedge^n(E),}}}
\end{equation}
with maps given by
\begin{equation}
\begin{gathered}
\delta_1([(u,v)])=[(u,v)]\otimes u,\qquad \delta_{n-1}([(u,v)]\otimes a)=u\wedge v\wedge a, \\\delta_k([(u,v)]\otimes a)=[(u,v)]\otimes u\wedge a\quad \text{for }1<k<n-1.
\end{gathered}
\end{equation}
Moreover, the map 
\begin{equation}
\widehat\B_{n-k}(\widehat F)\otimes\wedge^k(E)\to\B_{n-k}(\widehat F)\otimes\wedge^k(F^*),\qquad [(u,v)]\otimes a \to [r(u,v)]\otimes \pi_*(a)
\end{equation}
gives rise to a chain map $r\colon\widehat\Gamma(F,n)\to\Gamma(F,n)$.

The proofs of the next two theorems are significantly more involved.

\begin{theorem}[Proof in Section~\ref{sec:RnEqLnhatRels}]\label{thm:LnbetaEqZero}
If $\beta$ is in $\widetilde R_n(\C)$ then $\widehat\L_n(\beta)=0$ in $\C/\frac{(2\pi i)^n}{(n-1)!}\Z$. 
\end{theorem}

\begin{theorem}[Proof in Section~\ref{sec:Comparison}]\label{thm:LnhatEqLnOnBhat} If $\beta\in H^1(\widehat\Gamma(F,n))=\Ker(\delta_1)$ we have
\begin{equation}
\mathfrak R_n(\widehat\L_n(\beta))=\L_n(r(\beta)).
\end{equation}
\end{theorem}

\begin{conjecture}\label{conj:R2tildeR3tilde} $\widetilde R_2(F)$ is generated by realizations of (inverted) lifted five term relations. $\widetilde R_3(F)$ is generated by realizations of the 31 term relation.
\end{conjecture}

\begin{proposition}
Assuming Conjecture~\ref{conj:R2tildeR3tilde} for $\widetilde R_2(F)$ we have an isomorphism
\begin{equation}\label{eq:ResultSecB2Iso}
H^1(\widehat\Gamma(F,2))\cong\widehat\B(F)
\end{equation}
induced by $[(u,v)]\mapsto -[(v,u)]$. When $F=\C$, $\widehat\L_2$ agrees with Neumann's $R$.
\end{proposition}
\begin{proof}
The conjecture implies that the map induces an isomorphism between $\widehat\B_2(F)$ and $\widehat\P_E(F)$ (see Section~\ref{sec:ArbitraryFieldsIntro}). This proves the result.
\end{proof}

\subsection{Zero-degenerate points and permissible lifts}
Let $\alpha\in\Z[\widetilde S_1\times \widetilde S_1]$. Let $\overline I_\alpha$ be the ideal in $\Z[a_1,\dots,a_N]$ obtained from $I_\alpha$ by clearing denominators in~\eqref{eq:realizationIdealEqs}, and let $\overline{X_\alpha(F)}=\Spec(\Z[a_1,\dots,a_N]/\overline I_\alpha)$. For $q\in X_\alpha(F)$ we can still define $q(\alpha)\in\Z[P^1_F]$. If we introduce a formal variable $\log(0)$ we can also define a lift of $q$ with values in $E\oplus\Z[\log(0)]$.

\begin{definition}\label{def:ZeroDegLift} A point $q$ in $X_\alpha(F)$ is \emph{permissible} if either $q(z_i)$ or $q(w_i)$ is non-zero for all $i$. A permissible point is \emph{zero-degenerate} if $q(\alpha)\in\Z[\{0\}]$. A lift $\widetilde q$ of a zero-degenerate point $q$ is \emph{permissible} if $\widetilde q(v_i)=0$ when $z_i=0$ and $\widetilde q(u_i)=0$ when $w_i=0$.
\end{definition}

\begin{example}\label{eq:R2Permis}
Consider the element $\beta$ from Example~\ref{ex:InvFTSymb}. For any $a,b\in F\setminus\{0,1\}$, the point $(a_1,a_2,a_3,a_4,a_5)=(a,a,b,b,0)$ is zero-degenerate and $(\widetilde a,\widetilde a,\widetilde b,\widetilde b,\log(0))$ is a permissible lift whenever $\widetilde a$ and $\widetilde b$ in $E$ are lifts of $a$, respectively, $b$. 
\end{example}

\begin{example}\label{eq:R3Permis}
In Example~\ref{ex:31TermRel} a point with $a_i=0$, $b_J=1$ in $\overline{X_\alpha(F)}$ is zero-degenerate with a permissible lift $\widetilde a_i=\log(0)$, $\widetilde b_J=0$.
\end{example}

\subsection{Killing lower levels and proper ambiguity}
We start by defining the concepts for $n=2$. Consider the homomorphism
\begin{equation}\label{eq:NuSymbDef}
\delta\colon\Z[\widetilde S_1\times\widetilde S_1]\to\wedge^2(\widetilde S_1),\qquad (u,v)\mapsto u\wedge v.
\end{equation}
The following is elementary.
\begin{lemma}\label{lemma:L2AndNu}
An element $\alpha\in\Z[\widetilde S_1\times\widetilde S_1]$ is a differential $\widehat\L_2$ relation if and only if $2\delta(\alpha)=0$. If so, $\delta(\alpha)=0$ if and only if $\sum_{i=1}^Mk_{ji}l_{ji}$ is even for all $j=1,\dots, N$.
\end{lemma}
\begin{definition}
A differential $\widehat\L_2$ relation $\alpha$ has \emph{proper ambiguity} if $\sum_{i=1}^Mk_{ji}l_{ji}$ is even for all $j=1,\dots, N$
\end{definition}

\begin{remark} The definition was inspired by Neumann's parity condition~\cite[Def.~4.3]{Neumann}, the absence of which changes his map $R$ by 2-torsion.
\end{remark}

\begin{example}
The five term relations (Examples~\ref{ex:LiftedFTSymb} and \ref{ex:InvFTSymb}) have proper ambiguity.
\end{example}

When $n=2$ there are no lower levels, so all log-points kill the lower levels by default, and we can define $\widetilde R_2(F)$ as in Definition~\ref{def:RntildeDef}. When $n>2$ we first define the notion of \emph{lower level projections}.

\begin{definition} Let $l\in{2,\dots, n-1}$ and $J=\{j_1,\dots,j_{n-l}\}$ be a multisubset of $\{1,\dots,N\}$. The elements
\begin{equation}
\pi_J(\alpha)=\sum_{i=1}^M r_ik_{j_1i}\cdots k_{j_{n-l}i}(u_i,v_i)\in \Z[\widetilde S_1\times \widetilde S_1]
\end{equation}
are called \emph{level $l$ projections} of $\alpha$. We shall occasionally allow $l$ to be $n$ and define $\pi_\emptyset(\alpha)=\alpha$.
\end{definition}

\begin{definition}\label{def:KillLower} Let $n>2$ and let $\alpha$ be a differential $\widehat\L_n$ relation. A log-point $\widetilde p$ \emph{kills the lower levels of $\alpha$} if
\begin{equation}
\widetilde p(\pi_J(\alpha))\in \widetilde R_l(F)
\end{equation}
for all $l$ and all $J$ with $|J|=n-l$.
\end{definition}

\begin{definition}\label{def:ProperAmbiguity} Let $n>2$. A differential $\widehat\L_n$ relation $\alpha$ has \emph{proper ambiguity} if 
\begin{enumerate}
\item The element $\sum_{i=1}^M r_i(u_i\wedge v_i)\otimes u_i^{\otimes (n-2)}$ is zero in $\wedge^2(\widetilde S_1)\otimes \widetilde S_{n-2}$.
\item For all $j=1,\dots, N$ the integer $\sum_{i=1}^M k_{ij}^{n-1}l_{ij}$ is divisible by $n$.
\end{enumerate}
\end{definition}
Note that the first condition is always satisfied up to $2$-torsion (see Proposition~\ref{prop:SymbEqDiff}). When $n=2$ the two conditions coincide.
We can now define $\widetilde R_n(F)$ as in Definition~\ref{def:RntildeDef}.

\begin{remark} If one excludes the proper ambiguity condition from Definition~\ref{def:RntildeDef}, Theorems~\ref{thm:RntildeToRn} and~\ref{thm:Nuhat} still hold modulo 2-torsion, and Theorem~\ref{thm:LnbetaEqZero} holds modulo $(2\pi i)^n/n!$. Theorem~\ref{thm:LnhatEqLnOnBhat} still holds.
\end{remark}

\begin{example}
Goncharov's 22 term relation satisfies the second condition (all integers are either 0 or 3), but not the first. 
\end{example}

\begin{example}
The 31 term relation has proper ambiguity.
\end{example}

\begin{example}\label{ex:R31InR3} Each realization of an (inverted) five term relation is in $\widetilde R_2(F)$ by Example~\ref{eq:R2Permis}.  A straightforward (but tedious) calculation shows that all realizations of the lower level projections of the 31 term relation are linear combinations of lifted (inverted) 5 term relations. It thus follows from Example~\ref{eq:R3Permis} that all realizations of the 31 term relation are in $\widetilde R_3(F)$.
\end{example}

\subsection{Proof of Therems~\ref{thm:RntildeToRn} and~\ref{thm:Nuhat}}\label{sec:RnTildeDef}
We begin with the case $n=2$.

\begin{lemma}\label{lemma:AnWithn2} Let $\alpha$ be a differential $\widehat\L_2$ relation with proper ambiguity. For any smooth curve $Y$ in $X_\alpha(F)$, the element $p_Y(\alpha)$ is in $\A_{2}(F(Y))$.
\end{lemma}
\begin{proof}
The homomorphism $\pi\colon\widetilde S_1\to U$ induces a homomorphism $\pi_*\colon\wedge^2(\widetilde S_1)\to\wedge^2(U)$. Since each point $p\in Y$ restricts to a homomorphism $U\to F^*$, we have a homomorphism $p_{Y}\colon U\to F(Y)^*$.
We now have
\begin{equation}
\delta(p_Y(\alpha))=\sum_{i=1}^M r_i(\pi(u_i)_Y)\wedge (1-\pi(u_i)_Y)=\sum_{i=1}^M r_i\pi(u_i)_Y\wedge \pi(v_i)_Y=p_{Y*}\circ \pi_*(\delta(\alpha)).
\end{equation}
By Lemma~\ref{lemma:L2AndNu}, $\delta(\alpha)=0$. This proves the result.
\end{proof}

\begin{lemma}\label{lemma:ThmWithn2} The covering map $r$ takes $\widetilde R_2(F)$ to $R_2(F)$.
\end{lemma}
\begin{proof} 
Recall that $\widetilde R_2(F)$ is generated by the two types in Definition~\ref{def:RntildeDef}. Let $\alpha$ be a differential $\widehat\L_2$ relation with proper ambiguity and let $\widetilde p$ and $\widetilde q$ be log-points in the same geometric component of $X_\alpha(F)$. Pick a geometrically irreducible curve $Y$ containing $p$ and $q$ (such exists e.g.~by~\cite[Cor.~1.9]{FrancoisPoonen}). By Lemma~\ref{lemma:AnWithn2}, $p_Y(\alpha)\in\A_2(F(Y))$ and it follows that
\begin{equation}
r(\widetilde p(\alpha)-\widetilde q(\alpha))=p(\alpha)-q(\alpha)\in R_2(F).
\end{equation}
Similarly, if a curve $Y$ in $X_{K,L}(F)$ containing $p$ contains a zero-degenerate point $q\in \overline{X_\alpha(F)}$, we have $r(\widetilde p)=p(\alpha)\in R_2(F)$.
\end{proof}

Now suppose $n>2$. Assume by induction that $r$ maps $\widetilde R_k(F)$ to $R_k(F)$ for all $k<n$. We begin with an elementary lemma, which holds for any integer $m\geq 2$.

\begin{lemma}\label{lemma:RkOfExtension} Let $\beta\in\Z[\widetilde S_1\times\widetilde S_1]$ and let $\widetilde p$ be a log-point. Suppose that for a geometrically irreducible smooth curve $Y\subset X_\alpha(F)$ one has $p_Y(\beta)\in\A_m(F(Y))$.
If $p(\beta)\in R_m(F)$ for some point $p$ in $Y$, then $p_Y(\beta)\in R_m(F(Y))$.
\end{lemma}
\begin{proof}
Suppose $p_Y(\beta)=\sum_{i=1}^Mr_i[x_i]$, where $x_i\in F(Y)$. The extension of scalars $X$ of $Y$ to $F(Y)$ is a geometrically irreducible smooth curve in $X_\alpha(F(Y))$ containing both $q=(x_1,\dots,x_M)$ and $p$. Moreover, $p_{X}(\beta)\in\A_m(F(Y)(X))$. It follows that $p(\beta)-q(\beta)\in R_m(F(Y))$ and since $p(\beta)\in R_m(F)\subset R_m(F(Y))$, it follows that $q(\beta)=p_Y(\beta)\in R_m(F(Y))$.
\end{proof}

\begin{lemma}\label{lemma:LowerLevelDiff}
For $n>2$, $\alpha$ is a differential $\widehat\L_n$ relation if and only if
all level $l$ projections of $\alpha$ are differential $\widehat\L_l$ relations, which again holds if and only if all level 2 projections are differential $\widehat\L_2$ relations. A differential $\widehat\L_n$ relation satisfies the first proper ambiguity condition if and only if all level 2 projections have proper ambiguity.
\end{lemma}
\begin{proof}
This follows from the fact that the multiplication map $\Omega^1_{l-1}(\widetilde S)\otimes\widetilde S_{n-l}\to\Omega^1_{n-1}(\widetilde S)$ is an isomorphism for all $l$.
\end{proof}

\begin{lemma}\label{lemma:An} Let $\alpha$ be a differential $\widehat\L_n$ relation with proper ambiguity and $\widetilde p$ a log-point killing the lower levels of $\alpha$. For any smooth curve $Y$ in $X_\alpha(F)$ containing $p$, $p_Y(\alpha)\in\A_{n}(F(Y))$.
\end{lemma}
\begin{proof}
Our induction hypothesis that $r$ maps $\widetilde R_k(F)$ to $R_k(F)$ for all $k<n$ implies in particular that $p(\pi_{J_m}(\alpha))$ is in $R_m(F)$ for all $n-m$ element multisets $J_m$. Let us prove by induction on $m$ that
 \begin{equation}\label{eq:AnInduction}
 p_Y(\pi_{J_m}(\alpha))\in \A_m(F(Y))
 \end{equation}
 holds for all $m=2,\dots, n$ and all $J_m$. The case $m=n$ is the desired statement, and the case $m=2$ follows from Lemmas~\ref{lemma:LowerLevelDiff} and~\ref{lemma:AnWithn2}. Suppose by induction that~\eqref{eq:AnInduction} holds for $m=k>2$. For any $J_{k+1}$ we have
\begin{equation}
\delta(p_Y(\pi_{J_{k+1}}(\alpha)))=\sum_{j=1}^N p_Y(\pi_{J\cup\{j\}}(\alpha))\otimes p_Y(a_j),
\end{equation}
which by Lemma~\ref{lemma:RkOfExtension} is zero in $\B_{k}(F(Y))\otimes F(Y)^*$. Hence, $p_Y(\pi_{J_{k+1}}(\alpha))\in \A_{k+1}(F(Y))$.
 \end{proof}

\begin{theorem} The covering map $r$ takes $\widetilde R_n(F)$ maps to $R_n(F)$.
\end{theorem}
\begin{proof} 
The proof is the same as that of Lemma~\ref{lemma:ThmWithn2} except that Lemma~\ref{lemma:An} is used instead of Lemma~\ref{lemma:AnWithn2}.
\end{proof}

\begin{theorem}
The map $\delta$ from \eqref{eq:NunHatDef} takes $\widetilde R_n(F)$ to $0$.
\end{theorem}
\begin{proof}
Let $\alpha$ be a differential $\widehat\L_n$ relation and let $\widetilde p$ be a log-point killing the lower levels of $\alpha$. The lift $\widetilde p$ induces a homomorphism $\widetilde p_*\colon\wedge^2(\widetilde S_1)\to\wedge^2(E)$, and the result for $n=2$ now follows from the fact that 
$\widetilde p_*(\delta(\alpha))=\delta(\widetilde p(\alpha))$. For $n>2$, we have
\begin{equation}
\delta(\widetilde p(\alpha))=\sum_{i=1}^Kr_i[(\widetilde p(u_i),\widetilde p(v_i)]\otimes \widetilde p(u_i)=\sum_{j=1}^N\sum_{i=1}^K r_ik_{ji}[(\widetilde p(u_i),\widetilde p(v_i)]\otimes \widetilde p(\widetilde a_j).
\end{equation}
Since $\widetilde p$ kills the lower levels of $\alpha$, $\sum_{i=1}^K r_ik_{ji}[(\widetilde p(u_i),\widetilde p(v_i)]\in \widetilde R_{n-1}(F)$ for all $j$, and the result follows.
\end{proof}

\subsection{A variant using $\widehat F_{\pm}$}\label{sec:PlusMinusVariant} Now assume that $-1\neq 1\in F$.  By analogy with the case $F=\C$, we denote the unique element $e\in E$ with $2e=\iota(1)$ by $\pi i$.
%As we have seen, it is easier to give examples of realizations in $\widehat F_{\pm}$. %, and our results have natural analogues using $\widehat F_{\pm}$.
Since realizations in $\widehat F_{\pm}$ depend on a choice of sign determination, we need a notion of when two realizations are equivalent. %Since $-1\neq 1\in F$ there is a unique element $e\in E$ with $2e=\iota(1)$.
\begin{definition} Let $\mathcal V$ and $\mathcal V'$ be sign determinations and let $\widetilde p$ and $\widetilde q$ be log-points in $X_{K,L,\mathcal V}(F)$ and $X_{K,L,\mathcal V'}(F)$, respectively. We say that $\widetilde p$ and $\widetilde q$ are \emph{sign equivalent} if there is a point $r\in X_{K,L,\mathcal V}$ in the same geometric component as $p$ and a homomorphism $\phi\colon\widetilde S_1\to\Z$ such that $\widetilde q+\pi i\phi$ is a lift of $r$.
\end{definition}

\begin{example}
$[(\Log(x),\Log(1-x))]$ and $[(\Log(-y)-2\pi i,\Log(1-y)+2\pi i)]$ are sign equivalent realizations of $[(\widetilde a_1,\widetilde a_2)]$ in $\Z[\widehat\C_{\pm}]$.
\end{example}

%\begin{example}
%Changing a realization of a five term relation by a sign equivalence gives a five term relation in $\Z[\widehat\C_{\pm}]$.
%\end{example}

The notion of permissible lifts of zero-degenerate points is the same for points in $\overline{X_{\alpha,\mathcal V}(F)}$ and the notion of proper ambiguity is unchanged. The notion of killing lower levels now involves $\widetilde R_l(F)_{\pm}$. We can thus define $\widetilde R_n(F)_{\pm}$ as in Definition~\ref{def:RntildeDef}, but with $p_0\in X_{\alpha,\mathcal V}(F)$ and $\widetilde p$ sign equivalent to $\widetilde p_0$.
We can now define groups $\widehat\B_k(F)_{\pm}=\Z[\widehat F_{\pm}]/R_k(F)_{\pm}$, which fit in a chain complex $\widehat\Gamma(F,n)_{\pm}$.
The only thing that changes is that only $2\widetilde R_n(F)_{\pm}$ maps to $R_{n}(F)$, and that $\widehat\L_n$ only takes $\widetilde R_n(\C)_{\pm}$ to 0 modulo $\frac{(\pi i)^n}{(n-1)!}$. All proofs are identical to the case of $\widehat F$ (see Remark~\ref{rm:RppAmbiguity}).

\begin{remark}\label{rm:Gon22Lift}
One can show that if $\alpha\in\Z[\widetilde S_1\times\widetilde S_1]$ is Goncharov's 22 term relation, then any log-point $\widetilde p$ kills the lower levels of $2\alpha$. Since $2\alpha$ has proper ambiguity, it follows that $2\widetilde p(\alpha)$ is constant in $H^1(\widehat\Gamma(F,3)_{\pm})$. When $F=\C$ it follows from Theorem~\ref{thm:LnhatEqLnOnBhat} that $\widehat\L_3(\widetilde p(\alpha)))=3\zeta(3)$ modulo $\frac{(\pi i)^3}{4}$. With extra effort one can replace the denominator by $2$.
\end{remark}

%\begin{remark}
%Lift of Goncharov's relation. Comment on the 2-torsion.
%\end{remark}
%
%
%\begin{example}
%For $\alpha$ as in Example~\ref{ex:3term} we have
%\begin{equation}
%\pi_1(\alpha)=[(\widetilde a_1,\widetilde a_2)]-[(\widetilde a_2-\widetilde a_1,-\widetilde a_1)],\qquad\pi_2(\alpha)=-[(-\widetilde a_2,\widetilde a_1-\widetilde a_2)]+[(\widetilde a_2-\widetilde a_1,-\widetilde a_1)].
%\end{equation}
%For $\alpha$ as in Example~\ref{ex:Bad}, one has $\pi_{1,3}(\alpha)=4[(2\widetilde a_1+2\widetilde a_3,-\widetilde a_1+\widetilde a_2)]-8[(\widetilde a_1+\widetilde a_3,-2\widetilde a_1+\widetilde a_2-\widetilde a_3)]$. One easily verifies that these are indeed differential $\widehat\L_2$ relations.
%\end{example}

\section{Proof of Theorem~\ref{thm:LnbetaEqZero}}\label{sec:RnEqLnhatRels}
We now prove that $\widehat\L_n(\beta)=0$ if $\beta\in\widetilde R_n(\C)$. This is an immediate consequence of the following result together with Lemma~\ref{lemma:Cauchy}.

\begin{theorem}\label{thm:IndepOfLift} Let $\alpha$ be a differential $\widehat\L_n$ relation and $p$ a point in $X_\alpha(\C)$. If some lift of $p$ kills the lower levels of $\alpha$, then $\widehat\L_n(\widetilde p(\alpha))$ is independent of the choice of lift of $p$ modulo $\frac{(2\pi i)^n}{n!}$. If $\alpha$ has proper ambiguity, then this holds modulo $\frac{(2\pi i)^n}{(n-1)!}$.
\end{theorem}

In order to prove this result we begin with a technical lemma comparing the values of $\widehat\L_n$ at two points with the same image in $\C\setminus\{0,1\}$. Although our main interest is in $\widehat\C$ we formulate it for $\widehat\C_{\pm}$. The reason for the particular way of writing the right-hand side will become clear later.

\begin{lemma}\label{lemma:LhatDifference}  Let $\bar k=2\pi i k$ and $\bar l= 2\pi i l$ with $k,l\in\frac{1}{2}\Z$. We have
\begin{multline}\label{eq:LhatDifference}
\widehat\L_n(u+\bar k,v+\bar l)-\widehat\L_n(u,v)=\sum_{r=1}^{n-2}(-1)^r\frac{\bar k^r}{r!}\widehat\L_{n-r}(u,v)+\frac{(-1)^n}{n!}\Big(A(u,v;\bar k,\bar l)+\sum_{r=0}^{n-3}A_r(u,v;\bar k,\bar l)-
\bar k^{n-1}\bar l\Big),
\end{multline}
modulo $\frac{(\pi i)^n}{(n-1)!}$ where
\begin{equation}
A(u,v;\bar k,\bar l)=(\bar kv-\bar lu)(u+\bar k)^{n-2},\qquad A_r(u,v;\bar k,\bar l)=(\bar kv-\bar lu)\binom{n-2}{r+1}u^r\bar k^{n-2-r}.
\end{equation}
Moreover, if $(u,v)\in\widehat\C_{++}$ and $k,l\in\Z$, then~\eqref{eq:LhatDifference} holds modulo $\frac{(2\pi i)^n}{(n-1)!}$.
\end{lemma}
\begin{proof} Let's define functions $X$, $Y$ and $Z$ of $(u,v)\in\widehat\C_{\pm}$ as follows:\begin{equation}
\begin{gathered}
X=\frac{n!}{(-1)^n}\big(\widehat\L_n(u+\bar k,v+\bar l)-\widehat\L_n(u,v)\big),\quad Y=\frac{n!}{(-1)^n}\big(\sum_{r=1}^{n-2}\frac{(-1)^r}{r!}\bar k^r\widehat\L_{n-r}(u,v)\big)\\
Z=A(u,v;\bar k,\bar l)+\sum_{r=0}^{n-3}A_r(u,v;\bar k,\bar l)-\bar k^{n-1}\bar l
\end{gathered}
\end{equation}
%Note that when $(u,v)\in\widehat\C_{++}$ and $\bar k$ and $\bar l$ are even multiples of $\pi i$ these functions are defined modulo $n(2\pi i)^n$. Otherwise, we shall regard them as functions modulo $ n(\pi i)^n$.
 We must show that $X-Y=Z$. We first show that $dX-dY=dZ$. Using the fact that $d\widehat\L_k=\omega_k$, we obtain
\begin{equation}\label{eq:LHatDiff1}
\begin{aligned}
dX&=(n-1)\Big((u+\bar k)^{n-2}\big((u+\bar k)dv-(v+\bar l)du\big)-u^{n-2}(udv-vdu)\Big)\\
dY&=\sum_{r=1}^{n-2}\bar k^r\binom{n}{r}(n-r-1)u^{r-2}(udv-vdu).
\end{aligned}
\end{equation}
It follows that $dX-dY$ equals
\begin{equation}
\sum_{r=1}^{n-2}\Big(\big((n-1)\binom{n-2}{r}-\binom{n}{r}(n-r-1)\big)u^{n-r-2}\bar k^r\Big)(udv-vdu)+(n-1)(u+\bar k)^{n-2}(\bar kdv-\bar ldu).
\end{equation}
One now shows that $dX-dY=dZ$ by a term by term comparison. 
The coefficient of $dv$ in $dZ$ equals
\begin{equation}\label{eq:usdv}
\bar k(u+\bar k)^{n-2}+\sum_{r=0}^{n-3}\bar k\binom{n-2}{r+1}u^r\bar k^{n-2-r}.
\end{equation}
The coefficient of $u^sdv$ in $dX-dY$ is $n-1$ when $s=0$ and
\begin{equation}
(n-1)\binom{n-2}{n-1-s}-s\binom{n}{n-1-s}+(n-1)\binom{n-2}{s}=\binom{n-2}{s}+\binom{n-2}{s+1},
\end{equation}
when $s\in\{1,\dots,n-1\}$. By~\eqref{eq:usdv} this agrees with the coefficient of $u^sdv$ in $dZ$. A similar consideration comparing coefficients of $du$ completes the proof that $dX-dY=dZ$.

One now need only show that $X-Y-Z=0$ for a single point in each of the 4 components of $\widehat\C_{\pm}$. We choose points 
\begin{equation} (\pi i,\log(2)),\quad (0,\log(2)), \quad  (0,\log(2)+\pi i),\quad  (\pi i,\log(2)+\pi i)
\end{equation}
in $\widehat\C_{++}$, $\widehat\C_{-+}$, $\widehat\C_{--}$, $\widehat\C_{+-}$, respectively.
For $(u,v)=(\pi i,\log(2))$ we have
\begin{equation}
X=\frac{n!}{(-1)^n}\Bigg(\sum_{r=0}^{n-1}\frac{(-1)^r}{r!}\Li_{n-r}(-1)\big((\pi i+\bar k)^r-(\pi i)^r\big)\Bigg)-(\pi i+\bar k)^{n-1}(\log(2)+\bar l)+(\pi i)^{n-1}\log(2)
\end{equation}
and
\begin{equation}
\begin{aligned}
Y&=\frac{n!}{(-1)^n}\sum_{t=1}^{n-2}\frac{(-1)^t}{t!}\bar k^t\Big(\sum_{s=0}^{n-t-1}\frac{(-1)^s}{s!}\Li_{n-t-s}(-1)(\pi i)^s-\frac{(-1)^{n-t}}{(n-t)!}(\pi i)^{n-t-1}\log(2)\Big)\\
&=n\bar k^{n-1}\Li_1(-1)+\frac{n!}{(-1)^n}\Bigg(\sum_{r=1}^{n-1}\big((\pi i+\bar k)^r-(\pi i)^r\big)\frac{(-1)^r}{r!}\Li_{n-r}(-1)\Bigg)+\\&\hspace{2cm}-\sum_{t=1}^{n-2}\binom{n}{t}\bar k^t(\pi i)^{n-t-1}\log(2).
\end{aligned}
\end{equation}
Hence,
\begin{equation}\label{eq:Com1}
\begin{aligned}
X-Y&=(\pi i)^{n-1}\log(2)-(\pi i+\bar k)^{n-1}(\log(2)+\bar l)+\sum_{t=1}^{n-2}\binom{n}{t}\bar k^t(\pi i)^{n-t-1}\log(2)-n\bar k^{n-1}\Li_1(-1)\\
&=\log(2)\Big((\pi i)^{n-1}-(\pi i+\bar k)^{n-1}+\sum_{r=0}^{n-3}\binom{n}{r+2}\bar k^{n-2-r}(\pi i)^{r+1}+n\bar k^{n-1}\Big)-\bar l(\pi i+\bar k)^{n-1}.
\end{aligned}
\end{equation}
Similarly, we obtain
\begin{equation}\label{eq:Com2}
Z=\log(2)\Big(\bar k(\pi i+\bar k)^{n-2}+\sum_{r=0}^{n-3}\bar k \binom{n-2}{r+1}(\pi i)^r\bar k^{n-2-r}\Big)-\bar l(\pi i+\bar k)^{n-1}.
\end{equation}
Letting
\begin{equation}
a_r=\pi i\binom{n-2}{r+2}\bar k^{n-2-r}(\pi i)^r,\qquad b_r=\pi i\binom{n-2}{r+1}\bar k^{n-2-r}(\pi i)^r,
\end{equation}
\eqref{eq:Com1} and~\eqref{eq:Com2} together with the equality $\binom{n}{r+2}=\binom{n-2}{r+2}+2\binom{n-2}{r+1}+\binom{n-2}{r}$ imply that we have 
\begin{equation}
\begin{aligned}
X-Y-Z&=\log(2)\Big(n\bar k^{n-1}+\sum_{r=0}^{n-2}(a_{r}-a_{r-1})+2\sum_{r=0}^{n-2}(b_r-b_{r-1})\Big)\\&=\log(2)\Big(n\bar k^{n-1}-a_{-1}-2b_{-1}\Big)=0.
\end{aligned}
\end{equation}
This concludes the proof when $(u,v)=(\pi i,\log (2))\in\C_{++}$. The computations for the other 3 points are similar (and much simpler for $(0,\log(2))$ for $(0,\log(2)+\pi i)$).
\end{proof}

\begin{proof}[Proof of Theorem~\ref{thm:IndepOfLift}]
Now suppose $\alpha$ is a differential $\widehat\L_n$ relation with proper ambiguity and that $\widetilde p$ kills the lower levels of $\alpha$. 
For $j\in\{1,\dots,N\}$ let $T_j(\widetilde p)$ be the lift obtained from $\widetilde p$ by adding $2\pi i$ to $\widetilde a_j$. It is enough to prove that $\widehat\L_n(T_j(\widetilde p)(\alpha))=\widehat\L_n(\widetilde p(\alpha))$ for all $j$.
For an integer $k$, let $\bar k=2\pi i k$. By Lemma~\ref{lemma:LhatDifference}, we have
\begin{multline}
\widehat\L_n(T_j(\widetilde p)(\alpha))-\widehat\L_n(\widetilde p(\alpha))=\sum_{i=1}^M r_i\Bigg(\sum_{r=1}^{n-2}(-1)^r \frac{\bar k_{ji}^r}{r!}\widehat\L_{n-r}(\widetilde p(u_i),\widetilde p(v_i))\Bigg)+\\\frac{(-1)^n}{n!}\sum_{i=1}^M r_iA\Big(\widetilde p(u_i),\widetilde p(v_i);\bar k_{ji},\bar l_{ji}\Big)+\frac{(-1)^n}{n!}\sum_{i=1}^Mr_i \Bigg(\sum_{r=0}^{n-3} A_r\Big(\widetilde p(u_i),\widetilde p(v_i);\bar k_{ji},\bar l_{ji}\Big) -\\\frac{(-1)^n}{n!}\sum_{i=1}^M r_i (\bar k_{ij})^{n-1}(\bar l_{ij}).
\end{multline}
The first sum vanishes since $\widetilde p$ kills the lower levels of $\alpha$ (consider $J=\{j,\dots,j\}$). As we shall see, the second and third sum vanish since $\alpha$ is a differential $\widehat\L_n$ relation, and the last term vanishes by the (second) proper ambiguity condition. To see this consider the polynomial ring $\widetilde S[\widetilde a_{2\pi i}]$ obtained from $\widetilde S$ by adjoining a variable $\widetilde a_{2\pi i}$, which we think of as a symbolic representation of $2\pi i\in\C$. The homomorphism $\widetilde p\colon\widetilde S_1\to\C$ extends canonically to $\widetilde S[\widetilde a_{2\pi i}]_1$ by taking $\widetilde a_{2\pi i}$ to $2\pi i$. Note that $T_j(\widetilde p)(\alpha)=\widetilde p(T_j(\alpha))$, where $T_j\colon\widetilde S_1\to \widetilde S[\widetilde a_{2\pi i}]_1$ is the homomorphism taking $\widetilde a_j$ to $\widetilde a_j+\widetilde a_{2\pi i}$ and fixing all other generators. We have homomorphisms
\begin{equation}
\begin{aligned}
\chi&\colon\wedge^2(\widetilde S[\widetilde a_{2\pi i}]_1)\to\widetilde S_1, & (u+p\widetilde a_{2\pi i})\wedge (v+q\widetilde a_{2\pi i})\mapsto pv-qu\\
\wedge^2\Sym^{n-2}&\colon\Z[\widetilde S_1\times\widetilde S_1]\to\wedge^2(\widetilde S_1)\otimes\widetilde S_{n-2},& (u,v)\mapsto (u\wedge v)\otimes u^{n-2}
\end{aligned}
\end{equation}
as well as projection homomorphisms
\begin{equation}\Pi_k\colon \widetilde S[\widetilde a_{2\pi i}]_{n-2}\to\widetilde S_k
\end{equation} defined by taking a monomial $x$ to $x/\widetilde a_\pi^{n-2-k}$ if $x$ is divisible by $\widetilde a_\pi$ $n-2-k$ times and 0 otherwise.
%Note that
%\begin{equation}
%T_j(\alpha)=\sum_{i=1}^Kr_i(u_i+k_{ji}\widetilde a_\pi,v_i+l_{ji}\widetilde a_\pi).
%\end{equation}

Letting $m\colon\widetilde S[\widetilde a_{2\pi i}]_k\otimes \widetilde S[\widetilde a_{2\pi i}]_l\to \widetilde S[\widetilde a_{2\pi i}]_{k+l}$ be the multiplication map, the definition of the maps imply that
\begin{equation}\label{eq:ATerm}
m\circ(\chi\otimes\id)\circ T_{j*}(\wedge^2\Sym^{n-2}(\alpha))=\sum_{i=1}^{K}r_i(k_{ji}v_i-l_{ji}u_i)(u_i+k_{ji}\widetilde a_\pi)^{n-2}\in\widetilde S[\widetilde a_{2\pi i}]_{n-1}.
\end{equation}
By multiplication, $\widetilde p$ induces a homomorphism $\widetilde p\colon \widetilde S[\widetilde a_{2\pi i}]_{n-1}\to\C$, which takes the right-hand side of~\eqref{eq:ATerm} to $\sum_{i=1}^M r_iA\Big(\widetilde p(u_i),\widetilde p(v_i);\bar k_{ji},\bar l_{ji}\Big)$. Since $\alpha$ is a a differential $\widehat\L_n$ relation, the left-hand side of~\eqref{eq:ATerm} is $0$ since $\wedge^2\Sym^{n-2}(2\alpha)=0$. This proves the vanishing of the second sum.
The vanishing of the third sum is proved similarly, using that 
\begin{equation}
m\circ(\chi\otimes\id)\circ T_{j*}(\wedge^2\Sym^{n-2}(\alpha))=\sum_{i=1}^{K}r_i(k_{ji}v_i-l_{ji}u_i)(u_i+k_{ji}\widetilde a_\pi)^{n-2}\in\widetilde S[\pi]_{n-1},
\end{equation}
which holds for any $r$. 
Finally, the vanishing of the last sum follows from the second proper ambiguity condition. This concludes the proof of Theorem~\ref{thm:IndepOfLift}.
\end{proof}

\begin{remark}\label{rm:RppAmbiguity}
The same argument shows that $\widehat\L_n(\beta)=0$ modulo $\frac{(\pi i)^n}{(n-1)!}$ if $\beta$ is in the subgroup $\widetilde R(\C)_{\pm}$ defined in Section~\ref{sec:PlusMinusVariant}. The only difference is that adding a half integral multiple of $2\pi i$ to $\widetilde a_j$ changes the lift by a sign equivalence.
\end{remark}

%\begin{remark}\label{rm:n3}
%When $n=3$, independence of lift also holds under the slightly weaker assumption that $\widetilde p$ kills the lower levels of $2\alpha$ (PROPER AMBIGUITY 1). The proof of this is the same. Similarly, when it is enough that $\widetilde p$ is a proper realization of $q(n)\alpha$, where $q(n)$ is the greatest common divisor of the $\binom{n-2+r}{r}$ ($1\leq r\leq n-2$).
%\end{remark}

\section{Comparing $\widehat\L_n$ and $\L_n$}\label{sec:Comparison}
We now prove Theorem~\ref{thm:CompWithLnStatement} and Theorem~\ref{thm:LnhatEqLnOnBhat}. Recall that
\begin{equation}\label{eq:LnWithBeta}
\L_n(z)=\mathfrak R_n\Big(\sum_{r=0}^{n-1}\beta_r\Li_{n-r}(z)\Log(\vert z\vert)\Big),\qquad \beta_r=\frac{2^r}{r!}B_r.
\end{equation}
Define signs
\begin{equation}
\eta_j=\begin{cases}(-1)^{\frac{j(j-1)}{2}}&n\text{ even}\\(-1)^{\frac{j(j+1)}{2}}&n\text{ odd}\end{cases},\qquad \epsilon_j=\begin{cases}(-1)^{\frac{j}{2}}&j\text{ even}\\(-1)^{\frac{j+1}{2}}&j\text{ odd}.\end{cases}
\end{equation}
Let $\mathbf{1}_{\even}$ and $\mathbf{1}_{\odd}$ denote the characteristic functions for the even and odd numbers respectively. For non-negative integers $i$ and $j$ let
\begin{equation}\label{eq:cijdij}c_i=(1-2^{1-i})\beta_i,\qquad c_{i,j}=\frac{c_i}{j!}\eta_j,\quad d_{i,j}=(-1)^{\frac{i+2}{2}}\epsilon_n\sum_{r=0}^i\frac{c_r}{(i+j+2-r)!}\mathbf{1}_{\even}(i).
\end{equation}
In particular, $c_0=-1$. Note that up to a sign, the $c_{i,j}$ and $d_{i,j}$ are independent of $n$, and 0 when $i$ is odd.
We wish to prove that
\begin{multline}\label{eq:CompWithLnProofSection}
\mathfrak R_n(\widehat\L_n(u,v))-\L_n(z)=\sum_{s=1}^{n-2}\bigg(\mathfrak R_{n-s}(\widehat\L_{n-s}(u,v))\sum_{i=0}^s c_{i,s-i}\Real(u)^i\Imag(u)^j\bigg)\\+\det(u\wedge v)\sum_{i=0}^{n-2}d_{i,n-2-i}\Real(u)^i\Imag(u)^j,
\end{multline}
where $z=r(u,v)\in\C\setminus\{0,1\}$. We do this directly by expanding both sides of~\eqref{eq:CompWithLnProofSection} and comparing terms.
For notational simplicity let $\Li_r(z)=x_r+iy_r$ and $\Log(z)=a+bi$.  We shall only compare the terms involving $x_m$ and leave the analogous comparison of $y_m$ terms and terms not involving any $x_m$ or $y_m$ to the reader. Letting $\Coeff_{LHS}(x_m)$ and $\Coeff_{RHS}(x_m)$ denote the coefficients of $x_m$ when expanding the lefthand, respectively, righthand side of\eqref{eq:CompWithLnProofSection}, we thus wish to prove that $\Coeff_{LHS}(x_m)=\Coeff_{RHS}(x_m)$ for all $m$. We assume for notational simplicity that $(u,v)=(a+bi+2p\pi i,-x_1-iy_1+2q\pi i)\in\widehat\C_{++}$.

\subsection{The lefthand side}
As a simple consequence of the formula~\eqref{eq:LnhatIntro} we have
\begin{multline}\label{eq:LHS}
\widehat\L_n(u,v)=%\sum_{r=0}^{n-1}\Bigg(\frac{(-1)^r}{r!}\Big(x_n+iy_n-\frac{2\pi i q}{(n-r-1)!}(a+bi)^{n-r-1}\Big)(a+(b+2p\pi)i)^r\Bigg)-\\\frac{(-1)^n}{n!}(x_1-2q\pi i)(a+(b+2p\pi )i)^{n-1}\\=
\sum_{r=0}^{n-1}\frac{1}{r!}(-1)^rx_{n-r}\big(a+(b+2p\pi )i\big)^r+i\sum_{r=0}^{n-1}\frac{1}{r!}(-1)^ry_{n-r}\big(a+(b+2p\pi )i\big)^r\\
-(-1)^n\frac{n-1}{n!}x_1\big(a+(b+2p\pi)i\big)^{n-1}-i(-1)^n\frac{n-1}{n!}y_1\big(a+(b+2p\pi)i\big)^{n-1}\\-\frac{2q\pi i}{(n-1)!}(-2p\pi i)^{n-1}+\frac{2q\pi i}{n!}(-1)^{n-1}\big(a+(b+2p\pi)i\big)^{n-1}.
\end{multline}
Using~\eqref{eq:LHS} and \eqref{eq:LnWithBeta} we see that
\begin{equation}\label{eq:DiffZagierCoeffs}
\begin{aligned}
\Coeff_{LHS}(x_m)&=\frac{(-1)^{n-m}}{(n-m)!}\mathfrak R_n\Big((a+(b+2p\pi)i)^{n-m})\Big)-a^{n-m}\beta_{n-m}\mathbf{1}_{\odd}(n)\quad\text{for }1<m\leq n,\\
%y_m,\enspace m>1:\quad& \frac{(-1)^{n-m}}{(n-m)!}\mathfrak R_n\Big(i(a+(b+2p\pi)i)^{n-m})\Big)-a^{n-m}\beta_{n-m}\mathbf{1}_{\even}(n)\\
\Coeff_{LHS}(x_1)&= (-1)^{n-1}\frac{n-1}{n!}\mathfrak R_n\Big((a+(b+2p\pi)i)^{n-1})\Big)-a^{n-1}\beta_{n-1}\mathbf{1}_{\odd}(n).
%y_1:\quad& (-1)^{n-1}\frac{n-1}{n!}\mathfrak R_n\Big(i(a+(b+2p\pi)i)^{n-1})\Big)-a^{n-1}\beta_{n-1}\mathbf{1}_{\even}(n)\\
%\text{Remaining:}\quad &  (-1)^{n-1}\frac{2q\pi}{n!}\mathfrak R_n\Big(i(a+(b+2p\pi)i)^{n-1}\Big)
\end{aligned}
\end{equation}

%We use $\sum_{i+j=s}f(i,j)$ as shorthand for $\sum_{i=0}^sf(i,s-i)$.
%Replacing $s$ by $n-s$ in sum we get
\begin{lemma}\label{lemma:Rn} We have 
\begin{equation}\mathfrak R_s((a+(b+2p\pi)i)^{s-m})=\sum_{k+l=s-m}\binom{s-m}{k}(-1)^{\frac{m+k+1}{2}}\epsilon_s\mathbf{1}_{\odd}(m+k)a^k(b+2p\pi)^l.
\end{equation}
\end{lemma}
\begin{proof} This is an elementary consequence of the binomial theorem.
\end{proof}
It thus follows that $\Coeff_{LHS}(x_m)$ can be written as sums of terms of the form $a^k(b+2p\pi)^l$ where $k+l+m=n$. Let $\Coeff_{LHS}(x_m,k,l)$ denote the coefficient of $a^k(b+2p\pi)^l$ in $\Coeff_{LHS}(x_m)$. By Lemma~\ref{lemma:Rn} it follows from~\eqref{eq:DiffZagierCoeffs} that
\begin{equation}\label{eq:CLHSklm>1}
\Coeff_{LHS}(x_m,k,l)=\begin{cases}
\frac{(-1)^{n-m}}{k!l!}(-1)^{\frac{m+k+1}{2}}\epsilon_n\mathbf{1}_{\odd}(m+k)&\text{ for $l>0$}\\
\frac{(-1)^{n-m}}{k!l!}(-1)^{\frac{m+k+1}{2}}\epsilon_n\mathbf{1}_{\odd}(m+k)-\mathbf{1}_{\odd}(n)\beta_k&\text{ for $l=0$}
\end{cases}
\end{equation}
for $m>1$ and that
\begin{equation}\label{eq:CLHSklmEq1}
\Coeff_{LHS}(x_1,k,l)\begin{cases}
\frac{(-1)^{n-1}}{k!l!}\frac{n-1}{n}(-1)^{\frac{k+2}{2}}\epsilon_n\mathbf{1}_{\even}(k)&\text{ for $l>0$}\\
\frac{(-1)^{n-1}}{k!l!}\frac{n-1}{n}(-1)^{\frac{k+2}{2}}\epsilon_n\mathbf{1}_{\even}(k)-\mathbf{1}_{\odd}(n)\beta_k&\text{ for $l=0$.}
\end{cases}
\end{equation}

\subsection{The righthand side}
We shall need the following technical lemmas
\begin{lemma}\label{lemma:BinomLemma} For any non-negative integers $l$ and $s$ we have
\begin{equation}
\sum_{j=0}^l(-1)^j\binom{l}{j}=0 \text{ for }l>0,\qquad \sum_{j=0}^l\frac{(-1)^j\binom{l}{j}}{s+l-j}=\frac{(-1)^l}{s\binom{l+s}{l}}.
\end{equation}
\end{lemma}
\begin{proof}
The first is elementary and the second can be found in \cite[eq.~(5)]{BinomialIdentities}.
\end{proof}
\begin{lemma}\label{lemma:BetaLemma} Let $k$ and $l$ be non-negative integers with $l>1$ odd.
\begin{equation}
\sum_{i=0}^k\frac{c_i}{(k-i)!}=(-1)^{k-1}\beta_k,\qquad \sum_{i=0}^{l-1}\frac{l-1-i}{(l-i)!}c_i=-\beta_{l-1}
\end{equation}
\end{lemma}
\begin{proof}
Since $\sum_{r=0}^\infty \frac{B_r}{r!}x^r=\frac{x}{e^x-1}$ it follows that the generating function for $\beta_i$ is the function $f(x)=\frac{2x}{e^{2x}-1}$. The first equation now follows from the fact that $e^x(f(x)-2f(x/2))=-f(-x)$ and the second from the fact that $(\cosh(x)-\frac{\sinh(x)}{x})(f(x)-2f(x/2))=1-x-f(x)$.
\end{proof}

Since $(u,v)=(a+bi+2p\pi i,-x_1-iy_1+2q\pi i)$ it follows that 
\begin{equation}
\det(u\wedge v)=x_1(b+2p\pi)-y_1a+2\pi q a.
\end{equation}
The right hand side of~\eqref{eq:CompWithLnProofSection} expands to
\begin{equation}\sum_{s=2}^{n-1}\Big(\mathfrak R_{s}(\widehat\L_{s})\sum_{i+j=n-s} c_{i,j}a^i(b+2p\pi)^{j}\Big)+
(x_1(b+2p\pi)-y_1a+2\pi q a)\sum_{i+j=n-2} d_{i,j}a^i(b+2p\pi)^{j}
\end{equation}
and it follows from~\eqref{eq:LHS} that we have 
\begin{equation}\label{eq:CoeffRHS}
\begin{aligned}
\Coeff_{RHS}(x_m)&=\sum_{s=m}^{n-1}\bigg(\frac{(-1)^{s-m}}{(s-m)!}\mathfrak R_{s}\Big((a+(b+2p\pi)i)^{s-m})\Big)\sum_{i+j=n-s} c_{i,j}a^i(b+2p\pi)^{j}\bigg)\\
%y_m:&\quad\sum_{s=m}^{n-1}\bigg(\frac{(-1)^{s-m}}{(s-m)!}\mathfrak R_{s}\Big(i(a+(b+2p\pi)i)^{s-m})\Big)\sum_{\substack{i+j=n-s\\i\text{ even}}} c_{i,j}a^i(b+2p\pi)^{j}\bigg)\\
\Coeff_{RHS}(x_1)&=\sum_{s=2}^{n-1}\bigg((-1)^{s-1}\frac{s-1}{s!}\mathfrak R_{s}\Big((a+(b+2p\pi)i)^{s-1})\Big)\sum_{i+j=n-s} c_{i,j}a^i(b+2p\pi)^{j}\bigg)\\&\qquad+(b+2p\pi)\sum_{i+j=n-2} d_{i,j}a^i(b+2p\pi)^{j}.
%\begin{split}y_1:&\quad\sum_{s=2}^{n-1}\bigg((-1)^{s-1}\frac{s-1}{s!}\mathfrak R_{s}\Big(i(a+(b+2p\pi)i)^{s-1})\Big)\sum_{\substack{i+j=n-s\\i\text{ even}}} c_{i,j}a^i(b+2p\pi)^{j}\bigg)\\&\qquad-a\sum_{\substack{i+j=n-2\\i\text{ even}}} d_{i,j}a^i(b+2p\pi)^{j}\end{split}\\
%Remaining:&
\end{aligned}
\end{equation}
We can thus define $\Coeff_{RHS}(x_m,k,l)$ for $k+l+m=n$ as above. Let's first assume that $m>1$. By Lemma~\ref{lemma:Rn} $\Coeff_{RHS}(x_m,k,l)$ is given by
\begin{equation}
\begin{aligned}
\sum_{i=0}^k\sum_{\substack{j=0\\(i,j)\neq(0,0)}}^l\frac{(-1)^{n-i-j-m}}{(n-i-j-m)!}\binom{n-i-j-m}{k-i}(-1)^{\frac{m+k-i+1}{2}}\epsilon_{n-i-j}\mathbf{1}_{\odd}(m+k-i)c_{i,j}\\
=(-1)^{n-m}(-1)^{\frac{m+k+1}{2}}\epsilon_{n}\mathbf{1}_{\odd}(m+k)\bigg(\sum_{i=0}^k\sum_{j=0}^l\frac{c_i(-1)^{-j}\binom{l}{j}}{(k-i)!l!}-\frac{c_0}{k!l!}\bigg),
\end{aligned}
\end{equation}
which follows from the fact that $(-1)^{\frac{i}{2}}\epsilon_{n-i-j}=\epsilon_{n-j}=\epsilon_n\eta_j$ whenever $i$ is even. By Lemma~\ref{lemma:BinomLemma} and Lemma~\ref{lemma:BetaLemma} it follows that this agrees with~\eqref{eq:CLHSklm>1}. We have thus proved that $\Coeff_{LHS}(x_m)=\Coeff_{RHS}(x_m)$ for $m>1$. 

Now let $m=1$ and suppose $l>0$. By \eqref{eq:CoeffRHS} we see that $\Coeff(x_1,k,l)$ equals
\begin{equation}
\begin{aligned}
\MoveEqLeft[1]d_{k,l-1}+\sum_{\substack{i=0\\i\text{ even}}}^k\sum_{\substack{j=0\\(i,j)\neq(0,0)}}^l\frac{c_{i,j}(-1)^{n-j-1}(n-i-j-1)(-1)^{\frac{k+2}{2}}\epsilon_{n-j}\mathbf{1}_{\odd}(1+k)}{(n-i-j)(k-i)!(l-j)!}\\
&=d_{k,l-1}+(-1)^{n-1}(-1)^{\frac{k+2}{2}}\epsilon_{n}\mathbf{1}_{\even}(k)\bigg(\sum_{\substack{i=0\\i\text{ even}}}^k\sum_{j=0}^l\frac{c_i(-1)^{-j}(n-i-j-1)\binom{l}{j}}{(n-i-j)(k-i)!l!}-\frac{c_0(n-1)}{nk!l!}\bigg)\\
&=d_{k,l-1}+(-1)^{n-1}(-1)^{\frac{k+2}{2}}\epsilon_{n}\mathbf{1}_{\even}(k)\bigg(-\sum_{\substack{i=0\\i\text{ even}}}^k\sum_{j=0}^l\frac{c_i(-1)^{-j}\binom{l}{j}}{(n-i-j)(k-i)!l!}-\frac{c_0(n-1)}{nk!l!}\bigg)\\
&=d_{k,l-1}-(-1)^{n-1}(-1)^{\frac{k+2}{2}}\epsilon_{n}\mathbf{1}_{\even}(k)(-1)^l\sum_{\substack{i=0\\i\text{ even}}}^k\frac{c_i}{(n-i)!}+(-1)^{n-1}(-1)^{\frac{k+2}{2}}\epsilon_{n}\mathbf{1}_{\even}(k)\\
&=(-1)^{n-1}(-1)^{\frac{k+2}{2}}\epsilon_{n}\mathbf{1}_{\even}(k),
\end{aligned}
\end{equation}
where the second last equality follows from Lemma~\ref{lemma:BinomLemma}.
The fact that this equals~\eqref{eq:CLHSklmEq1} follows from Lemma~\ref{lemma:BetaLemma}.

Finally, when $l=0$ (so that $k=n-1$) a similar computation shows that $\Coeff_{RHS}(x_1,k,0)$ is given by
\begin{equation}
(-1)^{n-1}(-1)^{\frac{k+2}{2}}\epsilon_{n}\mathbf{1}_{\even}(k)\bigg(\sum_{\substack{i=0\\i\text{ even}}}^{n-1}\frac{n-1-i}{(n-i)!}c_i-\frac{c_0(n-1)}{nk!l!}\bigg).
\end{equation}
The fact that this agrees with \eqref{eq:CLHSklmEq1} follows from Lemma~\ref{lemma:BetaLemma}. This concludes the proof of Theorem~\ref{thm:CompWithLnStatement}.

\subsection{Proof of Theorem~\ref{thm:LnhatEqLnOnBhat}}
For $s=1,\dots,n-2$, let 
\begin{equation}\Psi_s\colon\Z[\widehat\C]\to\widehat\B_{n-s}(\widehat\C)\otimes \C^{\otimes s},\qquad [(u,v)]\mapsto [(u,v)]\otimes u^{\otimes s}.
\end{equation}
We note that 
\begin{equation}\label{eq:PsiFormula}
\Psi_s=\overbrace{(\delta\otimes \id)\circ\dots\circ(\delta\otimes \id)}^{s-1}\circ\delta.
\end{equation}
Also, let (for $i=0,\dots,s$)
\begin{equation}
\ReIm_i\colon\C^{\otimes s}\to\R,\qquad z_1\otimes\dots\otimes z_s\mapsto \prod_{k=1}^i\Real(z_k)\prod_{k=i+1}^s\Imag(z_k).
\end{equation}
Letting $c_{i,j}$ and $d_{i,j}$ be as above, define
\begin{equation}
C_s=\sum_{i=0}^sc_{i,s-i}\ReIm_i\colon\C^{\otimes s}\to\R,\qquad D=\sum_{i=0}^{n-2}d_{i,n-2-i}\ReIm_i\colon\C^{\otimes n-2}\to\R.
\end{equation}
Let
\begin{equation}
\Delta\colon\widehat\B_n(\widehat\C)\to\R,\qquad \alpha\mapsto\mathfrak R_n\circ\widehat\L_n(\alpha)-\L_n\circ r(\alpha).
\end{equation}
It then follows from Theorem~\ref{thm:CompWithLnStatement} that (where $m\colon\R\otimes\R\to\R$ is multiplication)
\begin{equation}
\Delta=m\circ \Big(\sum_{s=1}^{n-2}(\mathfrak R_{n-s}\circ\widehat\L_{n-s})\otimes C_s)\circ\Psi_s+(\det\circ D)\circ\Psi_{n-2}\Big).
\end{equation}
By~\eqref{eq:PsiFormula} this vanishes on $\Ker(\delta)$, and the result follows.

\section{A lift of Goncharov's regulator}\label{sec:LiftOfReg}
We begin with a review of Goncharov's results~\cite{GoncharovMotivicGalois} (revised in~\cite{GoncharovDeninger}; see also \cite{GoncharovRegulators,GoncharovArakelov}).
\subsection{Goncharov's regulator}\label{sec:ResultsGonReg}
For positive integers $2\leq p<q$, let $\widetilde{\Gr}(p,q)$ denote the affine cone over the Grassmannian of $p$-planes in $q$-space.  An element can be represented by a $p\times q$ matrix defined up to the action by $\SL(p)$ and we thus have an action of $S_q$ (the symmetric group on $q$ letters) on $\widetilde{\Gr}(p,q)$ obtained by permuting the columns of a representing matrix. 
For any $p$-element subset $I$ of $\{1,\dots,q\}$ we have a \emph{Pl\"ucker coordinate} $a_I$ defined as the $p\times p$ minor determined by $I$.
Let $\widetilde{\Gr}(p,q)^*$ denote the points where all Pl\"ucker coordinates are non-zero.
Goncharov showed that there is a commutative diagram
\begin{equation}\label{eq:GonDiagram}
\cxymatrix{{{\Z[\widetilde{\Gr}(3,7)^*(\C)]}\ar[r]^-{\partial}\ar[d]&{\Z[\widetilde{\Gr}(3,6)^*(\C)]}\ar[r]^-{\partial}\ar[d]^-{g_5}&{\Z[\widetilde{\Gr}(3,5)^*(\C)]}\ar[r]^-{\partial}\ar[d]^-{g_4}&{\Z[\widetilde{\Gr}(3,4)^*(\C)]}\ar[d]^-{g_3}\\0\ar[r]&\B_3(\C)_\Q\ar[r]^-{\delta}&(\B_2(\C)\otimes\C^*)_\Q\ar[r]^-{\delta}&\wedge^3(\C^*)_\Q}}
\end{equation}
where the boundary maps $\partial$ are the simplicial ones, and where
\begin{equation}
\begin{gathered}
g_3=\frac{1}{6}\Alt_4(a_{134}\wedge a_{124}\wedge a_{123}),\\
g_4=\frac{1}{12}\Alt_5\Big(\big[r(\overline v_1|\overline v_2,\overline v_3,\overline v_4,\overline v_5)\big]\otimes a_{123}\Big),\\
g_5=\frac{1}{90}\Alt_6([\frac{a_{124}a_{235}a_{136}}{a_{125}a_{236}a_{134}}\big]).
\end{gathered}
\end{equation}
Here, a quadruple of points in $(x_1,x_2,x_3,x_4)\in P^1_\C=\C\cup\{\infty\}$ has a \emph{cross-ratio} $\frac{(x_1-x_3)(x_2-x_4)}{(x_1-x_4)(x_2-x_3)}$, and $r(\overline v_1|\overline v_2,\overline v_3,\overline v_4,\overline v_5)$ denotes the cross-ratio of the projection of the quadruple $(v_2,v_3,v_4,v_5)$ to $P(\C^3/\langle v_1\rangle )=P^1_\C$. Also, $\Alt_n([x])=\sum_{\sigma\in S_n}\sgn(\sigma)[\sigma(x)]$. %We note that in~\cite{GoncharovDeninger}, $g_4$ has the opposite sign due to $\nu_2$ being defined as $(1-x)\wedge x$ and not $x\wedge (1-x)$.

Letting $G_q(p)=\Z[\widetilde{\Gr}(p,q+1)^*]$, there is a canonical map $\Gamma\colon H_*(\SL(p,\C))\to H_*(G_*(p))$, and Goncharov showed that the composition
\begin{equation}
\xymatrix{H_5(\SL(3,\C))\ar[r]^-{\Phi}&H_5(G_*(3))\ar[r]^-{g_5}&H^1(\Gamma(\C,3))_\Q\ar[r]^-{\L_3}&\R}
\end{equation}
is a non-zero rational multiple of the Borel regulator. Defining $\Gamma_i(\C,n)=\Gamma^{2n-i}(\C,n)$ one may view~\eqref{eq:GonDiagram} as a chain map $G_*(3)\to \Gamma_*(\C,3)_\Q$. 

\subsection{Cluster ensembles and differential $\widehat\L_n$ relations}
The varieties $\widetilde{\Gr}(p,q)$ are cluster ensembles in the sense of Fock and Goncharov~\cite{FockGoncharovClusterEnsembles}. The only thing we shall need about cluster ensembles is that there are two types of coordinates called $\A$-coordinates and $\X$-coordinates. The $\A$-coordinates are regular functions and include the Pl\"ucker coordinates. The $\X$-coordinates are monomial expressions in the $\A$-coordinates satisfying that if $X$ is an $\X$-coordinate, then $1+X$ has a canonical factorization as a monomial in the $\A$-coordinates. One may think of them as generalizations of cross-ratios.
 We note that $\widetilde{\Gr}(p,q)$ has finitely many $\A$ and $\X$ coordinates if and only if $(p-2)(p+q-2)<4$ (see e.g.~\cite{Scott,MotivicAmplitudes}).

If we regard an $\A$-coordinate as a formal variable, each $\X$-coordinate determines a generator $[X,1+X]$ of $\widetilde S_1\times\widetilde S_1$ (using multiplicative notation; see Section~\ref{sec:Examples}). Given a finite collection $C$ of $\X$-coordinates of $\widetilde{\Gr}(p,q)$, it becomes a simple linear algebra problem to determine if
\begin{equation}\label{eq:DiffClusterRel}
\alpha=\sum_{X\in C} r_X[X,1+X]\in\Z[\widetilde S_1\times\widetilde S_1]
\end{equation}
is a differential $\widehat\L_n$ relation. One always has the \emph{inversion relations}
\begin{equation}\label{eq:InversionRel}
[X,1+X]+(-1)^n[X^{-1},\frac{1+X}{X}]
\end{equation}
but when $n>3$ there seems to be no other relations of the form~\eqref{eq:DiffClusterRel}.

\begin{remark}\label{rm:C++Realizations}
Since we are primarily interested in realizations in $\Z[\widehat\C]$, we may consider $[\frac{X}{1+X},\frac{1}{1+X}]$ instead of $[X,1+X]$. We believe that this is in fact more natural.
\end{remark}

\begin{example}
For $\widetilde{\Gr}(3,6)$ there are 22 $\A$-coordinates, the 20 Pl\"ucker coordinates as well as 2 additional coordinates
\begin{equation}
y_1=\det(v_1\times v_2,v_3\times v_4,v_5\times v_6),\qquad y_2=\det(v_2\times v_3,v_4\times v_5,v_6\times v_1),
\end{equation}
where the $v_i$ are the columns of a representing matrix. There are $104$ $\X$-coordinates, which can all be obtained from the six $\X$-coordinates
\begin{equation}\label{eq:GeneratingXCoordinates}
\frac{a_{136}a_{235}}{a_{356}a_{123}},\quad \frac{a_{126} a_{145}}{a_{124} a_{156}},\quad \frac{a_{156} a_{236 }a_{345}}{a_{136} a_{235} a_{456}},\quad \frac{a_{123} a_{156}}{a_{126} a_{135}}\quad \frac{a_{136} a_{145} a_{235}}{a_{123} a_{156} a_{345}},\quad \frac{a_{123} a_{456}}{y_1}
\end{equation}
by inversion $x\mapsto x^{-1}$ and the action by the (dihedral) group generated by 
\begin{equation}
\sigma=(1,2,3,4,5,6)\in S_6,\qquad \tau=(1,6)(2,5)(3,4)\in S_6.
\end{equation}
Note that $\tau$ fixes $y_1$ and $y_2$ and $\sigma$ flips them. The number of elements in the $\langle\sigma,\tau\rangle$-orbits of the six $\X$-coordinates in~\eqref{eq:GeneratingXCoordinates} are 12, 12, 12, 6, 6, and 4, respectively.

For the $\X$-coordinates in~\eqref{eq:GeneratingXCoordinates}, $1+X$ is given by
\begin{equation}\frac{a_{135} a_{236}}{a_{123} a_{356}},\quad \frac{a_{125} a_{146}}{a_{124} a_{156}},\quad \frac{a_{356} y_2}{a_{136} a_{235} a_{456}},\quad\frac{a_{125} a_{136}}{a_{126}a_{135}},\quad\frac{a_{135} y_2}{a_{123} a_{156} a_{345}}, 
\quad \frac{a_{124} a_{356}}{y_1}.
\end{equation}

Up to the inversion relations~\eqref{eq:InversionRel} there are 25 linearly independent $\widehat\L_2$ relations and a single $40$ term differential $\widehat\L_3$ relation $R_{40}$. The 40 term relation is a lift of the 40 term relation for $\L_3$ found in~\cite{MotivicAmplitudes}. All lower levels are killed modulo 6-torsion.
\end{example}

\subsection{A lift of Goncharov's regulator}
Let $\widetilde{\Gr}(p,q)^{\A\neq 0}(\C)$ denote the points with non-zero $\A$-coordinates, and fix a branch of logarithm. Regarding the $\A$-coordinates as formal variables, and letting $X_{p,q}$ be the set of $\X$-coordinates of $\widetilde{\Gr}(p,q)$, we may regard an element of $\Z[X_{p,q}]$ either as an element in $\Z[\widetilde S_1\times\widetilde S_1]$ or as a map $\Z[\widetilde{\Gr}(p,q)^{\A\neq 0}(\C)]\to \Z[\widehat\C]$. Consider the element
\begin{equation}
\begin{split}\eta=\Alt_{\langle\sigma^2,\tau\rangle}\Big(\Big. \Big[\frac{a_{146}a_{245}}{a_{145}a_{246}},\frac{a_{124}a_{456}}{a_{145}a_{246}}\Big]+\Big[\frac{a_{124}a_{456}}{a_{145}a_{246}},\frac{a_{146}a_{245}}{a_{145}a_{246}}\Big]+\Big[\frac{a_{123}a_{146}a_{245}}{a_{124}y_2},\frac{a_{126}a_{145}a_{234}}{a_{124}y_2}\Big]\\
\phantom{ssssssss}+\Big[\frac{a_{126}a_{145}a_{234}}{a_{124}y_2},\frac{a_{123}a_{146}a_{245}}{a_{124}y_2}\Big]+\Big[\frac{a_{124}a_{156}}{a_{125}a_{146}},\frac{a_{126}a_{145}}{a_{125}a_{146}}\Big]+\Big[\frac{a_{126}a_{145}}{a_{125}a_{146}},\frac{a_{124}a_{156}}{a_{125}a_{146}}\Big]\Big.\Big)\\
\phantom{ss}-\Alt_{\langle\tau\rangle}\Big(\Big[\frac{a_{126}a_{234}a_{456}}{a_{246}y_2},\frac{a_{146}a_{236}a_{245}}{a_{246}y_2}\Big]+\Big[\frac{a_{146}a_{236}a_{245}}{a_{246}y_2},\frac{a_{126}a_{234}a_{456}}{a_{246}y_2}\Big]\Big).
\end{split}
\end{equation}

\begin{theorem} There is a commutative diagram
\begin{equation}\label{eq:GonDiagramLift}
\cxymatrix{{{\Z[\widetilde{\Gr}(3,7)^{\A\neq 0}(\C)]}\ar[r]^-{\partial}\ar[d]&{\Z[\widetilde{\Gr}(3,6)^{\A\neq 0}(\C)]}\ar[r]^-{\partial}\ar[d]^-{f_5}&{\Z[\widetilde{\Gr}(3,5)^*(\C)]}\ar[r]^-{\partial}\ar[d]^-{f_4}&{\Z[\widetilde{\Gr}(3,4)^*(\C)]}\ar[d]^-{f_3}\\0\ar[r]&{\widehat\B_3(\widehat\C)}\ar[r]^-{\delta}&{(\widehat\B_2(\widehat\C)\otimes\C)}\ar[r]^-{\delta}&{\wedge^3(\C)}}}
\end{equation}
with maps defined by
\begin{align}
\begin{split}
f_3&=\Alt_{\langle (1,2,3,4)\rangle}(\widetilde a_{134}\wedge \widetilde a_{124}\wedge \widetilde a_{123}),\\
f_4&=-\Alt_{\langle (1,2,3,4,5)\rangle}\Big(\big[\frac{a_{125}a_{134}}{a_{124}a_{135}},\frac{a_{123}a_{145}}{a_{124}a_{135}}\big]\otimes (\widetilde a_{123}+\widetilde a_{145})\Big),\\
\begin{split}f_5&=
\eta.
\end{split}
\end{split}
\end{align}

If $G_*^{\A\neq 0}(3)$ denotes the top chain complex, the composition
\begin{equation}
\xymatrix{{H_5(G_*^{\A\neq 0}(3))}\ar[r]^-{f_5}&{H^1(\widehat\Gamma(\C,3))}\ar[r]^-{r}&{\B_3(\C)_\Q}}
\end{equation}
agrees with Goncharov's map $g_5$.
\end{theorem}

\begin{proof} The proof that $\delta f_4=f_3\partial $ is elementary. We next show that $\delta f_5=f_4\partial$. We have
\begin{equation}
\delta f_5-f_4\partial=\sum A_a\otimes \widetilde a,
\end{equation}
where the sum is over the Pl\"ucker-coordinates and $y_2$. For example, we have
\begin{multline}
A_{a_{124}}=\Big[\frac{a_{124} a_{456}}{a_{145} a_{246}}, \frac{a_{146} a_{245}}{a_{145} a_{246}}\Big]-\Big[\frac{a_{123} a_{146} a_{245}}{a_{124} y_2}, \frac{a_{126} a_{145} a_{234}}{a_{124 }y_2}\Big]-\Big[\frac{a_{126} a_{145} a_{234}}{a_{124} y_2},\frac{a_{123} a_{146} a_{245}}{a_{124} y_2}\Big]\\+\Big[\frac{a_{124} a_{156}}{a_{125} a_{146}},\frac{a_{126} a_{145}}{a_{125} a_{146}}\Big]+[\frac{a_{123} a_{245}}{a_{124} a_{235}},\frac{a_{125} a_{234}}{a_{124} a_{235}}]+[\frac{a_{125} a_{234}}{a_{124} a_{235}}, \frac{a_{123} a_{245}}{a_{124} a_{235}}]\\+[\frac{a_{146} a_{245}}{a_{145} a_{246}}, \frac{a_{124} a_{456}}{a_{145} a_{246}}]+[\frac{a_{126} a_{145}}{a_{125} a_{146}}, \frac{a_{124} a_{156}}{a_{125} a_{146}}],
\end{multline}
which is a sum of instances of $[A,B]+[B,A]$. One can show (see Zickert~\cite[Rm.~8.7]{ZickertAlgK}) that $24([(u,v)]+[(u,v)])$ is a consequence of the lifted 5-term relations, so it follows that $A_{a_{124}}\otimes \widetilde a_{124}=0\in\widehat\B_2(\widehat\C)\otimes\C$. Similarly, $A_a\otimes\widetilde a$ is zero for the $\mathcal A$-coordinates $a_{125}, a_{134}, a_{256}, a_{346}, a_{356}, a_{135}, a_{246}$, and $y_2$. For the remaining terms we obtain
\begin{equation}
\delta f_5-f_4\partial=\Alt_{\langle \sigma^2,\tau\rangle}\big(A_{a_{123}}\otimes \widetilde a_{123}+A_{a_{136}}\otimes \widetilde a_{136}\big).
\end{equation}
One then checks that $A_{a_{123}}$ and $A_{a_{136}}$ are (up to instances of $[A,B]+[B,A]$) a sum of lifted five term relations. For example, we have $r(A_{a_{123}})=R_1+R_2+R_3$, where
\begin{equation}
\begin{aligned}
R_1&=[\frac{a_{123} a_{345}}{a_{134} a_{235}}]+[\frac{a_{125} a_{345}}{a_{135} a_{245}}]+[\frac{a_{125} a_{234}}{a_{124} a_{235}}]-[\frac{a_{124} a_{345}}{a_{134} a_{245}}]-[\frac{a_{125} a_{134}}{a_{124} a_{135}}]\\\\
R_2&=-[\frac{a_{126} a_{345}}{a_{136} a_{245}}]-[\frac{a_{126} a_{145} a_{234}}{a_{124} y_2}]+[\frac{a_{124} a_{345}}{a_{134} a_{245}}]-[\frac{a_{123} a_{146} a_{345}}{a_{134} y_2}]+[\frac{a_{126} a_{134}}{a_{124} a_{136}}]\\
	R_3&=-[\frac{a_{126} a_{135}}{a_{125} a_{136}}]+[\frac{a_{123} a_{156} a_{345}}{a_{135} y_2}]+[\frac{a_{126} a_{345}}{a_{136} a_{245}}]-[\frac{a_{125} a_{345}}{a_{135} a_{245}}]-[\frac{a_{123} a_{156} a_{245}}{a_{125} y_2}].\\
\end{aligned}
\end{equation}
It is not difficult to check that these are (up to instances of $[x]+[1-x]$) five term relations. This concludes the proof that $\delta f_5=f_4\partial$. 

We now prove that $f_5\partial=0$.
To see this we first compute $w_3(\eta)$ (recall Definition~\ref{def:wnalpha}). A straightforward computation shows that $w_3(\eta)=\Alt_{\langle \sigma\rangle}(\phi)$, 
where
\begin{equation}\label{eq:w3eta}
w_3(\eta)=\Alt_{\langle \sigma\rangle}(\phi),\qquad \phi=-\Alt_{\langle (1,2,3,4,5)\rangle}\Big(w_2(\big[\frac{a_{125}a_{134}}{a_{124}a_{135}},\frac{a_{123}a_{145}}{a_{124}a_{135}}\big]) (\widetilde a_{123}+\widetilde a_{145})\Big)
\end{equation}
From this we conclude that $w_3(f_5\partial)=0$. Since $\delta f_5=f_4\partial$ it follows that $f_5\partial$ kills lower levels modulo instances of $[A,B]+[B,A]$. In fact, these terms cancel out, so $f_5\partial$ kills lower levels. One easily checks that $f_5\partial$ has proper ambiguity, so $f_5\partial$ is constant in $\widehat\B_3(\widehat\C)$. One then checks that $f_5\partial$ vanishes on the nose
for the element
\begin{equation}\label{eq:GaugeMatrix}
\begin{pmatrix}
1&0&0&1&1&-2&1\\0&1&0&1&-2&1&-\frac{7}{8}\\0&0&1&\frac{1}{4}&1&1&-\frac{7}{2}
\end{pmatrix}.
\end{equation}
This concludes the proof that $f_5\partial=0$.
To see that $r(f_5)$ agrees with $g_5$ in homology, it is enough to prove that $\frac{1}{720}\Alt_6(r(\eta))=f_5\in\B_3(\C)_\Q$. 
Since $[x]+[1-x]=-[-\frac{1-x}{x}]+[1]\in\B_3(\C)$ we have 
\begin{equation}
r(\eta)=-\Alt_{\langle\sigma^2,\tau\rangle}\Big( \Big[-\frac{a_{124}a_{456}}{a_{146}a_{245}}\Big]+\Big[-\frac{a_{126}a_{145}a_{234}}{a_{123}a_{146}a_{245}}\Big]+\Big[-\frac{a_{126}a_{145}}{a_{124}a_{156}}\Big]\Big)+\Alt_{\langle\tau\rangle}\Big([-\frac{a_{146}a_{236}a_{245}}{a_{126}a_{234}a_{456}}\Big]\Big).
\end{equation}
The fact that $\frac{1}{720}\Alt_6(r(\eta))=f_5\in\B_3(\C)_\Q$ can now be verified by a term by term comparison, which does not use any relations in $\B_3(\C)_\Q$.
\end{proof}

\begin{remark}
It follows that $r(f_5)$ is equivalent to Goncharov's formula, but has the advantage of being defined with integral coefficients and without symmetrization.
\end{remark}

\begin{remark} It follows from~\eqref{eq:w3eta} that $\eta+\sigma(\eta)$ is a differential $\widehat\L_3$ relation with 80 terms. It has proper ambiguity and all lower levels are killed on the nose. It vanishes for the matrix~\eqref{eq:GaugeMatrix} with the last column removed, so is identically 0. From this it follows that $f_5$ is skew symmetric under the action by the dihedral group.
\end{remark}

\begin{remark}
The diagram~\eqref{eq:GonDiagramLift} may be defined over an arbitrary field. The righthand square always commutes, the middle square commutes modulo 24 torsion, and the left square commutes for any field where all $\A$-coordinates of~\eqref{eq:GaugeMatrix} are defined and non-zero.
\end{remark}

\bibliographystyle{alpha}
\bibliography{BibFile}

\end{document}